\documentclass[times,sort&compress,3p]{elsarticle}
\journal{Journal of Multivariate Analysis}
\usepackage[labelfont=bf]{caption}

\usepackage{amsmath,amsfonts,amssymb,amsthm,booktabs,color,epsfig,graphicx,hyperref,url}

\theoremstyle{plain}
\newtheorem{theorem}{Theorem}

\newtheorem{lemma}{Lemma}
\newtheorem{corollary}{Corollary}

\theoremstyle{definition}

\newtheorem{remark}{Remark}

\newcommand{\be}{\begin{eqnarray}}
\newcommand{\ee}{\end{eqnarray}}
\newcommand{\bea}{\begin{eqnarray*}}
\newcommand{\eea}{\end{eqnarray*}}

\newcommand{\ND}{\textcolor{black}}
\DeclareMathOperator*{\argmax}{arg\,max}

\newcommand{\cond}{\stackrel{\mathcal{D}}{\to}}
\newcommand{\conp}{\stackrel{\textnormal{Pr}}{\to}}
\newcommand{\tr}{\operatorname{tr}}

\newcommand{\inv}{^{-1}}

\newcommand{\PR}{\textnormal{Pr}}
\newcommand{\bfx}{\mathbf{x}}
\newcommand{\bfy}{\mathbf{y}}
\newcommand{\iid}{\stackrel{\textnormal{i.i.d.}}{\sim}}

\newcommand{\bfSigma}{\mathbf{\Sigma}}
\newcommand{\bfX}{\mathbf{X}}
\newcommand{\diag}{\operatorname{diag}}
\newcommand{\bfI}{\mathbf{I}}

\newcommand{\lb}{\left(}
\newcommand{\rb}{\right)}

\newcommand{\N}{\mathbb{N}}

\newcommand{\R}{\mathbb{R}}

\newcommand{\E}{\mathbb{E}}
\newcommand{\bfb}{\mathbf{b}}
\newcommand{\bfP}{\mathbf{P}}

\begin{document}

\begin{frontmatter}

\title{Likelihood ratio tests under model misspecification in high dimensions}

\author[1]{Nina D\"ornemann \corref{mycorrespondingauthor}}

\address[1]{Fakultät für Mathematik, Ruhr-Universität Bochum, 44801 Bochum, Deutschland}

\cortext[mycorrespondingauthor]{Corresponding author. Email address: \url{nina.doernemann@rub.de}}

\begin{abstract}
We investigate the likelihood ratio test for a large block-diagonal covariance matrix with an increasing number of blocks under the null hypothesis. While so far the likelihood ratio statistic has only been studied for normal populations, we establish that its asymptotic behavior is invariant under a much larger class of distributions. This implies
robustness against model misspecification, which is common in high-dimensional regimes. Demonstrating the flexibility of our approach, we additionally establish asymptotic normality of the log-likelihood ratio test statistic for the equality of many large sample covariance matrices under model uncertainty. For this statistic, a subtle adjustment to the centering term is needed compared to normal case.
A simulation
study and an analysis of a data set from psychology emphasize the usefulness of our findings.
\end{abstract}


\end{frontmatter}

\section{Introduction} \label{sec_intro}
Over the last decades, the availability of high-dimensional data sets across diverse disciplines such as biostatistics, wireless communications and finance has transformed statistical practice 
(see, e.g., \cite{Fan2006, Johnstone2006} and references therein). Traditional multivariate analysis, as outlined in the text book \cite{anderson2003, muirhead1982}, is developed under the paradigm that the dimension is
negligible compared to the sample size and breaks down seriously if this assumption is violated. Such problems have spurred the development of new analysis tools, that work for dimensions of the same order as and even larger than the sample size. The literature on these topics is so large, that we can only cite a few illustrative examples, related to the present work: The works \cite{yamadaetal2017, bodnar2019} 
address whether a large covariance matrix admits a block-diagonal structure. Tests for independence in various setting are discussed in \cite{han2017distribution, loubaton2020large}. 
In the work \cite{hu2017testing}, Hu et al. concentrate on tests for the equality of high-dimensional mean vectors, while \cite{he2021asymptotically} take a broader perspective on high-dimensional testing by investigating a class of $U$-statistics. 

Turning closer to the scope of this work, the likelihood ratio method has received much attention in the literature on high-dimensional statistical inference since the past decade.
The starting point for the investigating of various classical testing problems transferred to a high-dimensional setting can be seen in the work of \cite{jiang_yang_2013} establishing CLTs for the corresponding log-likelihood ratio tests, including the two main testing problems investigated in this work. The authors of \cite{jiang2015} tried to relax the assumptions on the parameters, while other authors extended these results in various directions. For example, \cite{jiang2017moderate} proved a moderate deviation principle for these likelihood ratio tests.
More recent generalizations include the works of \cite{qi_et_al_2019, dette2020, guo2021}. 
All of these works rely on normally distributed data, and the asymptotic behavior of these test statistics under model misspecification has received little attention in the literature on high-dimensional statistics so far. A few works investigating likelihood ratio tests in different settings under model uncertainty include \cite{luo2012proportional, lemonte2013gradient, lemonte2016gradient, strug2018evidential, ishii2021statistical}. 
We add to this line of literature by dropping the restrictive distributional assumption of normality. In particular, we find that the CLTs for the log-likelihood of two specific testing problems remain still valid when only assuming moments of order $(4+\delta)$ for some $\delta>0$. Besides the theoretical importance of our findings, these results ensure more robust statistical guarantees for practitioners, as the validity of the normal assumption is not a priori clear for high-dimensional data sets. 

Interestingly, our results reveal that the limiting null distribution of the log-likelihood for testing for a block-diagonal structure does not depend on specific characteristics of the underlying data generating distribution, such as the fourth moment. 
 This observation is illustrated in Fig. \ref{fig_intro} 
 where we consider the problem of testing whether the covariance matrix of  a $p$-dimensional
random vector admits a block-diagonal structure with $q$ blocks.
Here, we display three histograms for the corresponding log-likelihood ratio test statistic
 under the null hypothesis based on three different distributions for the samples. The components of all vectors are independent identically distributed with respect to a standard normal distribution (left column), standardized t-distribution (middle column) and centered exponential distribution (right column), respectively. Thus, the null hypothesis of a block-diagonal covariance matrix is obviously satisfied (in this case, the covariance matrix equals the identity matrix). We observe that the histograms look very similar. This testing problem will be examined in detail in Section \ref{sec_test_uncor} of this work.  
  \begin{figure}[!ht]
      \includegraphics[width=0.329\columnwidth]{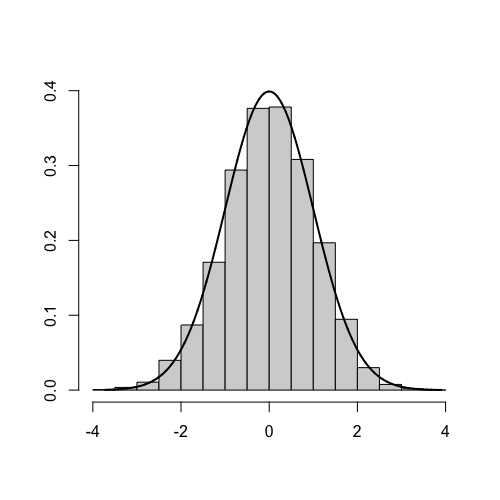}
 \includegraphics[width=0.329\columnwidth]{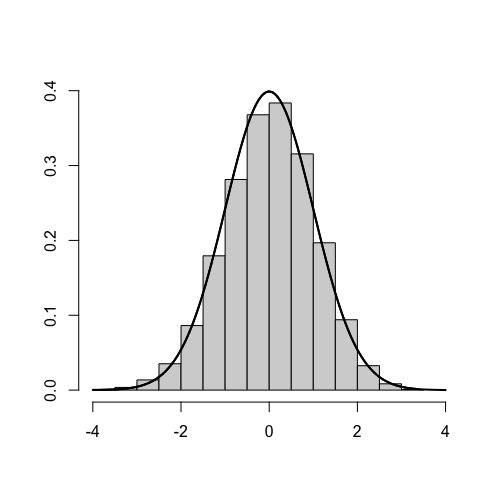}  
 \includegraphics[width=0.329\columnwidth]{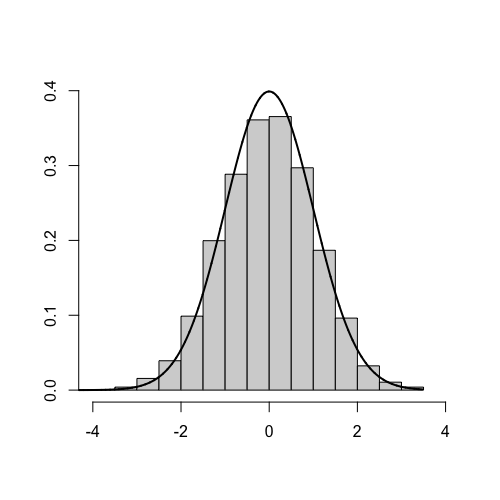}  
    
    \caption{ Histograms for the log-likelihood ratio test statistic \eqref{def_statistic} under the hypothesis \eqref{null} of a block diagonal structure of $q=30$ blocks
    of equal size $2$  in a $p=60$-dimensional
 random vector (sample size  $100$, simulation runs $10,000$). Left column: standard normally distributed data, middle column: standardized t-distributed data with $15$ degrees of freedom, right column: centered exponentially distributed data with parameter $1$. The gray curve indicates the density of the standard normal distribution. 
       }\label{fig_intro}
\end{figure} 
\\ Before concluding this introduction, we would like to discuss the main ideas for our proofs. Under the normal assumption, the exact distribution of the test statistic is available under the null hypothesis, on which proofs of previous works crucially depend on. Equipped with such a knowledge, the moment-generating of the log-likelihood test statistic is investigated \citep[see, e.g.,][]{jiang_yang_2013, qi_et_al_2019, guo2021} or a general central limit theorem is applied  \citep[see][]{dette2020}. Obviously, we cannot hope for an analogue exact representation without knowing the underlying distribution of the data. 

In order to tackle the difficulties arising in the proof for non-normal populations, we 
 derive a novel representation of the log-likelihood test statistic involving random quadratic forms without imposing restrictive distributional assumptions. For this purpose, we perform a QR-decomposition for the (sub)data matrices. Such QR-decompositions are useful in a broader context: The authors of \cite{wang2018} used this tool in order to derive the logarithmic law of the sample covariance matrix for the case $p/n \to 1$ near singularity, while \cite{heiny2021} investigated the log-determinant of the sample correlation matrix under infinite fourth moment. These papers were partially inspired by works of  \cite{nguyen_vu, bao2015}, in which the authors proved Girko's logarithmic law for a general random matrix with independent entries and brought his “method of perpendiculars” \citep[see][]{girko1998refinement} to a mathematically rigorous level. 
 Via our representation, we are in the position to decompose the test statistic into three parts: we will prove that the dominating linear term satisfies a central limit theorem for martingale difference schemes, while the quadratic term converges to constant and the remainders are asymptotically negligible.
 Heuristically, this decomposition can be motivated by Taylor's expansion $\log(1+x) = x - x^2/2 + \mathcal{O}(x^3)$, though one needs more delicate arguments in order to justify this step mathematically correct.  
 
 This work is structured as follows. In Section \ref{sec_test_uncor}, we present a CLT for the log-likelihood ratio test of a block-diagonal covariance matrix under the null hypothesis. Here, the number of blocks may increase together with the dimension of the data and sample size while we do not assume that the data is generated by a normal distribution. As a corollary, the distribution of a test for a diagonal covariance matrix is derived. 
 In Section \ref{sec_test_eq_cov}, we apply our method to another classical likelihood ratio test and provide the asymptotic distribution of the log-likelihood ratio test on the equality of many large covariance matrices. The main results of Section \ref{sec_test_uncor} and \ref{sec_test_eq_cov} are proven in Section \ref{sec_proofs}.
 We illustrate our findings with a simulation study, including a comparison to other criteria, and a real data analysis in Section \ref{sec_sim}. 
 
\section{Testing for a block-diagonal covariance matrix} \label{sec_test_uncor}

In the main part of this work, we revisit a very prominent problem in high-dimensional data analysis, namely a test for uncorrelation of  sub-vectors of a multivariate
 distribution. For normally distributed data, this coincides with a test for independence of these sub-vectors. 
 To be precise, let  $\mathbf{y} = \bfSigma^{\frac{1}{2}}\mathbf{x}$ denote a  $p$-dimensional random vector with mean
 $\boldsymbol\mu = \mathbf{0}\in\mathbb{R}^{p}$  and covariance matrix  $\mathbf{\Sigma}\in\mathbb{R}^{p \times p}$, where $\bfSigma^{\frac{1}{2}}$ denotes the symmetric square root of $\bfSigma$. We assume that  $\mathbf{y}$ is decomposed
 as
	\begin{align} \label{notation}
		\mathbf{y} = \big (  {\mathbf{y}^{(1)}}^{\top}	, 	\ldots ,
		{\mathbf{y}^{(q)}}^{\top}
		\big ) ^{\top},
		\end{align}	
		where $\mathbf{y}^{(i)}$ are vectors of dimension  $p_i\in\N$, $1 \leq i \leq q$, such that $\sum_{i=1}^{q} p_{i}=p$ for some
		 integer $q\geq 2$. Moreover, we assume that the components of $\bfx$ are i.i.d. with respect to some centered and standardized distribution.
		Let
		\begin{align} \label{22}
		\mathbf{\Sigma} = \begin{pmatrix}
			\mathbf{\Sigma}_{11} & \mathbf{\Sigma}_{12} & \ldots  & \mathbf{\Sigma}_{1q} \\
			\mathbf{\Sigma}_{21} & \mathbf{\Sigma}_{22}& \ldots  & \mathbf{\Sigma}_{2q} \\
			\vdots & \vdots & \ddots & \vdots\\
			\mathbf{\Sigma}_{q 1} & \mathbf{\Sigma}_{q 2} & \ldots  & \mathbf{\Sigma}_{q q} 
		\end{pmatrix}
	\end{align}	
		denote the corresponding decomposition of the covariance matrix, where $\mathbf{\Sigma}_{ij} := \operatorname{Cov}(\mathbf{y}^{(i)},\mathbf{y}^{(j)})$. The hypothesis of uncorrelated sub-vectors is formulated as
		\begin{align} \label{null}
		H_0: ~ \mathbf{\Sigma}_{ij}  =\mathbf{0} \textnormal{~~for all~} i\neq j.
	\end{align}	
Let $\bfx_1, \ldots, \bfx_n \iid \bfx$ be a sample  of independent identically distributed random variables according to $\bfx$ and denote $\bfy_k = \bfSigma^{1/2} \bfx_k$ for $1\leq k \leq n$. 
Under the normal assumption $\bfx \sim \mathcal{N} ( \mathbf{0}, \bfI )$,
 the likelihood ratio test is given by 
	\begin{align} \label{def_statistic}
		\Lambda_n = \frac{|\mathbf{\hat \Sigma}|^{n/2}}{\prod\limits_{i=1}^{q} |\mathbf{\hat \Sigma}_{ii} |^{n/2}} = V_n^{\frac{n}{2}}, ~
		\mathbf{\hat \Sigma} =\frac{1}{n}\sum\limits_{k=1}^n \mathbf{y}_k 
		\mathbf{y}_k^\top,
	\end{align}
where $\hat \bfSigma$ denotes the sample covariance matrix of $\bfy_1, \ldots, \bfy_n$ and
and $ \mathbf{\hat \Sigma}_{{ij} }$ denotes the block in the $i$th row and $j$th column of  the estimate $\mathbf{ \hat \Sigma}$ corresponding to
the decomposition \eqref{22}. \\
The authors of \cite{jiang_yang_2013, jiang2015} derived a central limit theorem for the corresponding log-likelihood ratio test statistic assuming that $\bfy \sim \mathcal{N}(\boldsymbol\mu, \bfSigma)$ and that the number $q$ of blocks is fixed. Several authors such as \cite{qi_et_al_2019, dette2020} demonstrated that such a CLT still holds true if the parameter $q=q_n$ is allowed to increase with sample size and dimension. All of these works rely on normally distributed data. 
Dropping the normal assumption, the following theorem provides the asymptotic distribution of $\log \Lambda_n$ under the null hypothesis of uncorrelation without assuming a normal distribution for $\bfx_1, \ldots, \bfx_n \sim \bfx$. The proof is deferred to Section \ref{sec_proofs}. 
\begin{theorem} \label{thm}
Let the components of $\bfx$ be i.i.d. centered random variables following a continuous distribution with finite $(4+\delta)$th moment for some $\delta>0$. 	
Assume that $q =q_n \geq 2$ is a possibly increasing integer, $2 \leq p=p_n< n$ with  $0< \inf_{n\in\N} \min_{1 \leq i \leq q} ( p_i q)/n \leq \sup_{n\in\N} p/n < 1$
 and $\max_{1 \leq i \leq q} p_i \leq \eta p$ for each $n\in\N$ and some fixed $\eta \in (0,1)$.  Then, it holds under the null hypothesis in \eqref{null} that
	\begin{align*}
		\frac{ \log V_n - \mu_n }{\sigma_n} \cond \mathcal{N}(0,1),
			\end{align*}
			where 
			\begin{align*}
				\mu_n & = \sum\limits_{i=1}^q \lb n - p_i - \frac{1}{2} \rb \log \lb 1 - \frac{p_i}{n} \rb  - \lb n - p - \frac{1}{2} \rb \log \lb 1 - \frac{p}{n} \rb 
				, ~ \sigma_n^2  = 2 \left\{ \sum\limits_{i=1}^q \log \lb 1 - \frac{p_i}{n} \rb 
				- \log \lb 1 - \frac{p}{n} \rb \right\} 
				.
			\end{align*}
\end{theorem}

\begin{remark} \label{rem_substitution} 
\begin{enumerate}[(i)]
    \item 
Choose $\alpha \in (0,1)$. We propose to reject the null hypothesis whenever
       \begin{align}
			 \log V_n \leq \sigma_{n} u_{\alpha} + \mu_n~, \label{test}
		\end{align}			        
        where $u_{\alpha}$ denotes the $\alpha$-quantile of the standard normal distribution. Thus, we have by Theorem \ref{thm}
        \begin{align*}
        		\lim\limits_{n\to\infty} \PR_{H_0} \lb \log V_n \leq \sigma_{n} u_{\alpha} + \mu_n \rb 
        		= \PR ( \mathcal{N}(0,1) \leq u_\alpha ) =\alpha,
        \end{align*}
        which means that the test keeps asymptotically its nominal level $\alpha$.
        \item 
        \ND{
        We would like to comment on the continuity assumption on the distribution of the random vector $\bfx$, which is needed in the proof of Theorem \ref{thm} in order to ensure formal correctness of the QR-decomposition. In general, one could replace the entries of $\mathbf{x}$ by some random vector $\overline{\bf{x}}$ following a continuous distribution without changing the asymptotic distribution of $\log V_n$. One choice for $\overline{\bfx}$ could be a mixture of the original distribution and a uniform distribution (see, e.g., Proposition 2.1 in \cite{wang2018} for the log-determinant of the sample covariance matrix).  
        If $q_n=q$ denotes a fixed integer, the continuity assumption is not necessary for proving Theorem \ref{thm}. 
        For an increasing parameter $q_n$, however, this would require a more restrictive regime for $q = q_n$, since the log-likelihood test statistics involves $q_n+1$ log-determinants of sample covariance matrices. 
        We do not pursue in this direction here.
        } 
        \end{enumerate} 
\end{remark}

As a noteworthy-by-product of Theorem \ref{thm}, we are able to construct a test for a diagonal covariance matrix based on the sample correlation matrix. 
For this purpose, we consider the special case of testing for a diagonal covariance matrix which coincides with complete independence of the $p$ components of $\bfx$ in the normal case. In this case, the test in \eqref{null} is equivalent to 
\begin{align}
\label{null_diag_cov}
H_0: \mathbf{R} = \mathbf{I}_p,
\end{align}
where $
\mathbf{R} = \diag(\bfSigma)^{-\frac{1}{2}} \bfSigma  \diag(\bfSigma)^{-\frac{1}{2}}
$ denotes the population correlation matrix of $\bfy$. Then, the statistic $V_n$ defined in \eqref{def_statistic} can be written as the determinant of the sample correlation matrix, that is,
\begin{align*}
	V_n = | \hat{\mathbf{R}} |,
\end{align*}
where 
$
	\hat{\mathbf{R}} = \diag(\hat{\bfSigma})^{-\frac{1}{2}} \hat{\bfSigma } \diag(\hat{\bfSigma})^{-\frac{1}{2}}
$
denotes the sample correlation matrix of $\bfy_1, \ldots, \bfy_n$. Several authors investigated tests for the hypothesis given in \eqref{null_diag_cov} in different frameworks (e.g., see \cite{jiang_yang_2013, jiang2015, gao2017, mestre2017correlation, qi_et_al_2019, parolya2021, heiny2021}).
We observe that testing for \eqref{null_diag_cov} is a special case of testing for \eqref{null} by letting $q=p$ and $p_1 = \cdots = p_q = 1$.
Then, Theorem \ref{thm} gives us the following result.
\begin{corollary} \label{cor}
Let the components of $\bfx$ be i.i.d. centered random variables following a continuous distribution with finite $(4+\delta)$th moment for some $\delta>0$. 
Assume that $2 \leq p=p_n< n$ and $0< \inf_{n\in\N} p/n \leq \sup_{n\in\N} p/n < 1$.  Then, it holds under the null hypothesis in \eqref{null_diag_cov} that,
	\begin{align*}
		\frac{ \log | \mathbf{R} | - \overline{\mu}_n }{\overline{\sigma}_n} \cond \mathcal{N}(0,1),
			\end{align*}
			where 
			\begin{align*}
				\overline{\mu}_n & = p \lb n -  \frac{3}{2} \rb \log \lb 1 - \frac{1}{n} \rb  - \lb n - p - \frac{1}{2} \rb \log \lb 1 - \frac{p}{n} \rb ,
				~ \overline{\sigma}_n^2  = 2 \left\{ p \log \lb 1 - \frac{1}{n} \rb 
				- \log \lb 1 - \frac{p}{n} \rb \right\} 
				. 
			\end{align*}
\end{corollary}
 Note that \cite{parolya2021} investigated the log-determinant of the sample correlation matrix in a more general context. In fact, they cover the case $\mathbf{R} \neq \mathbf{I}$, while Corollary \ref{cor} is formulated under the null hypothesis $\mathbf{R} = \mathbf{I}$. 
If we assume that $p/n \to \gamma \in (0,1)$, then Corollary \ref{cor} yields a special version of their Theorem 2.1, since
	\begin{align*}
		\overline{\mu}_n & = - \lb n - p - \frac{1}{2} \rb \log \lb 1 - \frac{p - 1 }{n} \rb 
		- ( p - 1) +  \frac{p}{n} + o(1), ~ 
		\overline{\sigma}_n^2  = - 2 \left\{ \frac{p}{n} - \log \lb 1 - \frac{p- 1 }{n} \rb \right\} + o(1),
	\end{align*}
which coincides with mean and variance given in their Theorem 2.1 in the case $\mathbf{R}= \mathbf{I}$. 
In a follow-up work, \cite{heiny2021} showed that the CLT for the sample correlation matrix in the case $\mathbf{R}=\mathbf{I}$ still holds true under infinite fourth moment. 
\ND{
Other authors were interested in the study of the spatial-sign covariance matrix, which includes the sample correlation matrix as a special case. For example, \cite{yang2021, li2022} proved CLTs for linear spectral statistics of this general class of random matrices.
}

	\section{Testing for equality of covariance matrices} \label{sec_test_eq_cov}
	We expect that our method for proving a CLT as given in Theorem \ref{thm} can be adapted to the investigation of other classical likelihood ratio tests in a non-normal setting.
	 In order to demonstrate this adaption, we consider in this section the comparison of $q$  centered distributions with covariance matrices
   $\mathbf{\Sigma}_{1} , \ldots , \mathbf{\Sigma}_{q}
   \in\R^{p\times p}$ and generic elements $\bfy_1 =\bfSigma_1^{1/2} \bfx, \ldots, \bfy_q = \bfSigma_q^{1/2} \bfx$.
    We assume that for each group $j$
   a sample of size $n_j$ is available, $j\in\{1,\ldots ,q\}$. When considering asymptotics, the dimension $p$ and the number $q$ of groups increase with the (sub)sample sizes.
   As before, we assume that the components of $\bfx$ are i.i.d. with respect to some centered and standardized distribution.

An important assumption for  multivariate analysis of variance (MANOVA)  is that  of  equal covariances in the different groups, which motivates our interest in testing the hypothesis
 	\begin{align} \label{null_eq_cov}
 		H_0: \mathbf{\Sigma}_1 = \cdots  = \mathbf{\Sigma}_q.
 	\end{align}
	This problem has been considered by several authors in the context of high-dimensional inference
	(see  \cite{obrien1992, schott2007, srivastava2010, jiang_yang_2013, jiang2015, dette2020, guo2021} among others).
	In this section, we add to this line of literature and
	investigate the asymptotic distribution of  the likelihood ratio test based on samples of independent distributed observations
	$\mathbf{y}_{jk} \iid  \bfy_j,$ $1\leq k \leq n_j$, $1\leq j \leq q$. To be precise,   let
$n = \sum_{j=1}^q n_j $ 		be the total sample size,
then the likelihood ratio test for the hypothesis \eqref{null_eq_cov} under the normal assumption $\bfx\sim\mathcal{N} (\mathbf{0}, \bfI)$
is  given by
	\begin{align} \label{def_statistic_2}
		\Lambda_{n,2} = \frac{\prod\limits_{j=1}^q |\mathbf{A}_j / n_j |^{\frac{1}{2} n_j }}{|\mathbf{A}/n|^{\frac{1}{2}n }},
	\end{align}
	where  the $p \times p$  matrices   $\mathbf{A}_{j}$ and $\mathbf{A}$ are defined as
	\begin{align*}
		\mathbf{A}_j &= \sum_{k=1}^{n_j} \mathbf{y}_{jk} \mathbf{y}_{jk}^\top  ~,~~
		\mathbf{A} = \sum_{j=1}^q \mathbf{A}_j. 
	\end{align*}	 
In the case $\bfy_j\sim\mathcal{N}(\boldsymbol\mu_j, \bfSigma_j)$, $1 \leq j \leq q$,	 \cite{jiang_yang_2013, jiang2015} proved asymptotic normality of the corresponding log-likelihood ratio test statistic under the null hypothesis if the number $q$ of groups is fixed. These results were generalized by \cite{dette2020, guo2021} for the case of an increasing number $q=q_n$ of groups. All of these works dealt only with normally distributed data, while the following result shows that a subtle adjustment to the centering term is necessary for non-normal data.
	In the following theorem, we provide the limiting distribution of $\log \Lambda_{n,2}$ under the null hypothesis without imposing a normal assumption on $\bfy_1, \ldots, \bfy_q$ in a high-dimensional setting, where the number of groups is allowed to increase. 
\begin{theorem} \label{thm_eq_cov}
Let the components of $\bfx$ be i.i.d. centered random variables following a continuous distribution with finite $(4+\delta)$th moment for some $\delta>0$. 
Assume that $q=q_n\geq2$ is a possibly increasing integer, $n_j =n_j(n) > p=p_n $ for every $n\in\N$ and $0< \inf_{n\in\N} p/n \leq \sup_{n\in\N} \max_{1 \leq j \leq q} p/n_j < 1$. 
Then it holds under the null hypothesis \eqref{null_eq_cov}
\begin{align*}
	\frac{2 \lb \log \Lambda_{n,2}  - \mu_n \rb }{n \sigma_n} \cond \mathcal{N}(0,1),
\end{align*}
where
\begin{align}
	\mu_n  & =  n \lb n - p - \frac{1}{2} \rb \log \lb 1 - \frac{p}{n} \rb 
	- \sum\limits_{j=1}^q n_j \lb n_j - p - \frac{1}{2} \rb \log \lb 1 - \frac{p}{n_j} \rb 
+ \frac{\E[x_{11}^4] -3}{2} p(1-q)
	, ~ \nonumber \\ 
	\sigma_n^2 & =  \log \lb 1 - \frac{p}{n} \rb  - \sum\limits_{j=1}^q \lb \frac{n_j}{n} \rb^2 \log \lb 1 - \frac{p}{n_j} \rb 
	. \label{def_sigma_eq_cov}
\end{align}

\end{theorem}
The proof is provided is Section \ref{sec_proofs}. Similarly to Remark \ref{rem_substitution}, an asymptotic level $\alpha$ test for the hypothesis \eqref{null_eq_cov} can be constructed using Theorem \ref{thm_eq_cov}. For the sake of brevity, we omit the details. 

\section{Finite-sample properties}\label{sec_sim}

 \begin{figure}[!ht]
     \centering
  \includegraphics[width=0.3\columnwidth]{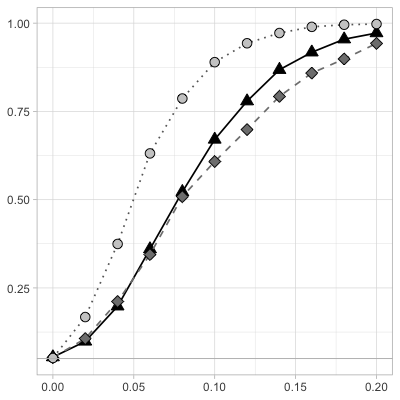}  \hspace{0.5cm}
   \includegraphics[width=0.3\columnwidth]{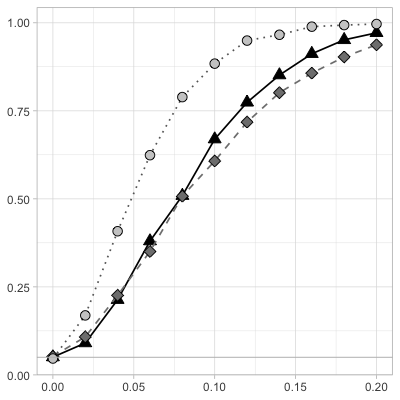}   \hspace{0.5cm}
    \includegraphics[width=0.3\columnwidth]{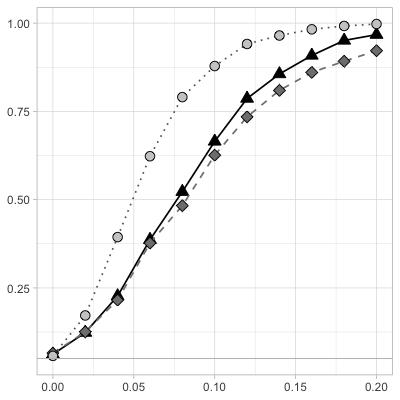}   
    
\vspace{0.5cm}
    
   \includegraphics[width=0.3\columnwidth]{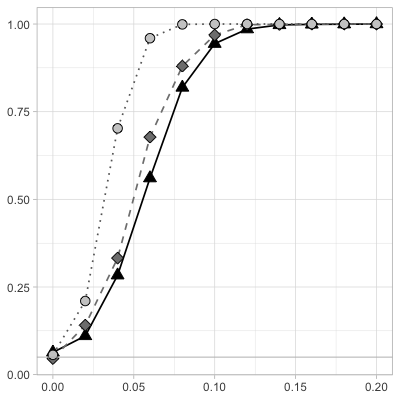} \hspace{0.5cm}
  \includegraphics[width=0.3\columnwidth]{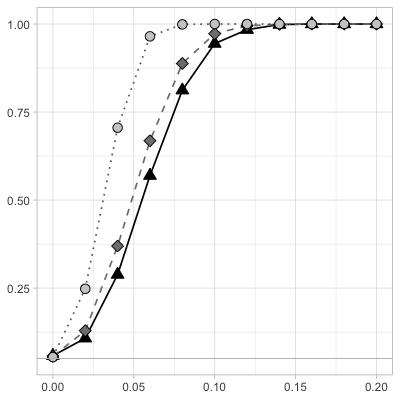}   \hspace{0.5cm}
 \includegraphics[width=0.3\columnwidth]{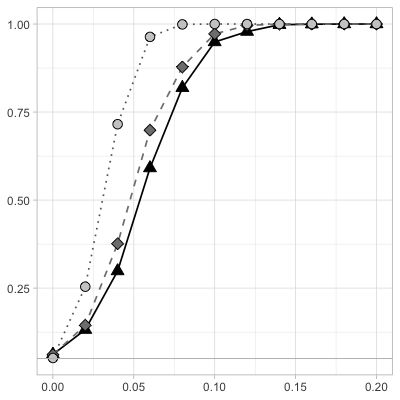}  
   
    \caption{ Rejection rates for the log-likelihood ratio test \eqref{test} for Scenario (i) (first row) and Scenario (ii) (second row) under  \eqref{alternative} based on $2,000$ simulation runs. First column: standard normally distributed data, second column: standardized t-distributed data with $15$ degrees of freedom, third column: centered exponentially distributed data with parameter $1$. The triangle indicates $n=100, p=60$, the square $n=120, p=90$ and the circle $n=180, p=120$. 
 The vertical grey line in each figure defines the nominal level $\alpha = 5 \%$.
       }\label{fig_power}
       
\end{figure}

In this section, we investigate the finite-sample properties of the test \eqref{test} under both the hypothesis and the alternative. Following \cite{qi_et_al_2019}, we consider the following alternative 
 \begin{align}
 \label{alternative}
 \bfSigma  = ( 1 - \delta) \bfI + \delta \mathbf{1},
 \end{align}
 where $\mathbf{1}$  denotes the $p\times p$ matrix filled with ones and $\mathbf{I}$ denotes the $p\times p$ identity matrix. Here, the parameter $\delta\geq 0$ determines the “distance” to the null hypothesis \eqref{null} (note that the choice $\delta =0$ corresponds to the null hypothesis \eqref{null}).  
In Fig. \ref{fig_power}, we display the empirical rejection rates of the test \eqref{test} for different choices of $\delta$, $n$, $p$, $q$, $p_i$ and different distributions for the random vector $\bfx$.
	All results are based on $2,000$ simulation runs and the components of $\bfx$ are independent identically distributed with respect to a standard normal distribution (first column), standardized t-distribution (second column) and centered exponential distribution (third column), respectively. The vertical gray line in each figure defines the nominal level $\alpha = 5 \%$. For the choice of the different groups, we consider the following two scenarios:
	\begin{enumerate}[(i)]
	\item $q=3$, $p_1 = p_2 = p_3 =p/3$,
	\item $q=p/2$, $p_1 = \cdots = p_{q-1}=1, p_q = q+1 $. 
	\end{enumerate}
	We observe a good approximation of the nominal level in all cases under consideration. 
	Moreover, the power increases reasonably as $\delta$ increases. It should be noted that the increase in power is a bit stronger for Scenario (ii) than for Scenario (i).
	The finite-sample properties do not significantly differ for the three underlying data generating distributions, as indicated by the asymptotic result provided in Theorem \ref{thm}.
	Overall, the test admits a desirable performance for finite-sample sizes, both for a large number and relatively small number of groups as covered by Scenarios (i) and (ii), and the accuracy improves for large sample size and dimension.

\subsection{Comparison to other tests} \label{sec_comp}
In this section, we compare the test given in \eqref{test} to two trace criteria for the null hypothesis \eqref{null} of blockdiagonality. Both procedures rely on the normal assumption for the data. In this case, substitution principles are widely available, and the tests are based on the centered sample covariance matrix
	\begin{align*}
		\mathbf{\hat \Sigma}^{\textnormal{cen}} =\frac{1}{n}\sum\limits_{k=1}^n \lb \mathbf{y}_k - \overline{\bfy} \rb 
		\lb \mathbf{y}_k - \overline{\bfy} \rb ^\top
	\end{align*}
and its blocks $\bfSigma_{ij}^{\textnormal{cen}}$ corresponding to the decomposition in \eqref{22}, where
\begin{align*}
    \overline{\mathbf{y}} = \frac{1}{n} \sum_{k=1}^n \bfy_k
\end{align*}
denotes the sample mean. 
On the one hand, we consider the tests proposed by \cite{bao2017test, jiang2013testing} who assume that $q\geq 2$ is fixed. On the other hand, we investigate a criterion proposed by \cite{li2017testing} for $q=2$ groups.
To make a meaningful comparison, we restrict our analysis to the case $q=2$ with normally distributed data. In this case, the test statistics proposed by \cite{bao2017test, jiang2013testing} both are equal to 
\begin{align*}
    L_n = \operatorname{tr} \lb \hat\bfSigma_{21}^{\textnormal{cen}} (\hat\bfSigma_{11}^{\textnormal{cen}})\inv \hat\bfSigma_{12}^{\textnormal{cen}} (\hat\bfSigma_{22}^{\textnormal{cen}}) \inv \rb ,
\end{align*}
which is shown to satisfy 
\begin{align*}
    T_1 = \frac{L_n - m_n  }{\sqrt{v}} \cond \mathcal{N}(0,1)
\end{align*}
under $H_0$ given in \eqref{null} \citep[see Theorem 3.2 in the work of][]{jiang2013testing}. 
Here, we set
\begin{align*}
    m_n = \frac{p_2 r_{n2} }{r_{n1} + r_{n2}},
    v = \frac{2 h^2 r_1^2 r_2^2 }{(r_1 + r_2)^4} , ~
    h = \sqrt{r_1 + r_2 - r_1r_2}
\end{align*}
and assume that $r_{n1} = p_2 / p_1 \to r_1 \in (0,\infty)$ and $r_{n2} = p_2 / (n - 1 - p_1) \to r_2 \in (0,\infty), p_2 <n.$ \\ 
Coming back to the test investigated by \cite{li2017testing}, we define
\begin{align*}
    \hat\gamma_{ij}= \frac{1}{(n-2) (n+1)}
    \left\{ \tr \lb \hat\bfSigma_{ij}^{\textnormal{cen}} \hat\bfSigma_{ji}^{\textnormal{cen}}  \rb 
    - \frac{1}{n - 1} \tr ( \hat\bfSigma_{ii}^{\textnormal{cen}} ) \tr (\hat\bfSigma_{jj}^{\textnormal{cen}}) \right\}, 
    ~ i,j\in \{1,2\}.
\end{align*}
    If $p\to\infty$ as $n\to\infty$, the authors show that 
    \begin{align*}
        T_2 = \sqrt{\frac{(n-2) (n+1)}{2}} \frac{\hat\gamma_{12}}{\sqrt{\hat\gamma_{11} \hat\gamma_{22}}}
        \cond \mathcal{N}(0,1)
    \end{align*}
    under $H_0$ in \eqref{null} assuming that
    \begin{align*}
        0 < \lim_{n\to\infty} \tr \bfSigma^k  < \infty, ~
        k \in\{ 1,2,4\}.
    \end{align*}
    For a prescribed level $\alpha \in (0,1),$ we reject $H_0$ if $T_2 > u_{1-\alpha} $ (or $T_1 > u_{1-\alpha} $ when using the test based on $T_1$). For simplicity of presentation, we refer to the test decisions based upon the statistics $T_1$ and $T_2$ also by $T_1$ and $T_2$. 
       We expect the test $T_2$ to admit a strong rejection rate under the alternative \eqref{alternative} since the core part $\hat\gamma_{12}$ of the statistic $T_2$ compares the trace of the product of the two off-diagonal blocks to the product of traces of the two diagonal blocks.  This difference is expected to be large under the alternative \eqref{alternative}. 
    In Table \ref{table}, the empirical rejection rates of $T_1, T_2$ and the log-likelihood ratio test LLRT (see \eqref{test}) under the null hypothesis \eqref{null} and the alternative \eqref{alternative} are displayed. 
    Additionally, we consider the alternative model 
    \begin{align} \label{model}
    \bfy^{(1)} = (1+\delta) \mathbf{z}^{(1)}, ~ \bfy^{(2)} = \mathbf{z}^{(2)} + \delta  \mathbf{z}^{(1)}  
    \end{align} 
    (recall the notation in \eqref{notation}), where $\mathbf{z}^{(i)} \sim \mathcal{N}_{p_i} (\mathbf{0}, \mathbf{I})$ are independent for $i\in \{1,2\}.$ Note that the choice $\delta = 0$ in \eqref{model} corresponds to the null hypothesis \eqref{null}. 
    This model has been considered in previous works \citep[see, e.g.,][]{qi_et_al_2019}. 
    The corresponding empirical rejection rates can be found 
    in Table \ref{table2}. 
    In most cases, the nominal level $\alpha = 5\% $ is approximated well by all tests under the null hypothesis. Under the alternative \eqref{alternative}, the test $T_2$ outperforms $T_1$ as well as the LLRT in terms of empirical power. The LLRT still performs reasonably well under this alternative, while the performance of $T_1$ is not satisfying in most cases. 
      Under model \eqref{model}, we observe for all three tests a strong rejection rate under the alternative with $T_2$ leading the field. 
      
      In summary, we observe that $T_2$ is slightly preferable in the cases under consideration, 
      which is not surprising as $T_2$ is just tailored to the case of two groups with normally distributed data, while the user may apply the LLRT in a broader context. 
      However, even in this restrictive situation, the LLRT still forms a reliable criterion with a reasonable rejection rate.  
    We emphasize that the theoretical statistical guarantees for $T_1$ and $T_2$ rely on normally distributed data and a fixed number $q$ of groups, while the LLRT is shown to satisfy a CLT beyond the normal assumption for a possibly increasing number of groups (Theorem \ref{thm}).

    \renewcommand{\arraystretch}{1.4}
    \begin{table} 
    \caption{\label{table} Empirical rejection rates of the tests $T_1, T_2$ and the LLRT under the null hypothesis \eqref{null} and the alternative \eqref{alternative} for different values of $n,p_1,p_2, q=2$, $\mathbf{x} \sim \mathcal{N}(\mathbf{0}, \mathbf{I})$ and $2,000$ simulation runs. }
    \begin{tabular}{c ccc ccc ccc}
 & & $\delta = 0$ & & & $\delta = 0.15$ & & &  $\delta = 0.3$ & 
\\
\cline{2-10}
   $(n, p, p_1)$ & $T_1$ & $T_2$ & LLRT 
   & $T_1$ & $T_2$ & LLRT
   & $T_1$ & $T_2$ & LLRT \\
   $(150, 80, 70)  $ & 0.048 & 0.041 & 0.051
    & 0.529 &  1.000 & 0.787
    & 0.717 & 1.000 & 0.995  \\
    $(150,80,40)$  & 0.054 & 0.054 & 0.059
 & 0.418 & 1.000 & 0.838
 & 0.518 & 1.000 & 0.995\\
    $(150,100,80) $      & 0.057 & 0.060 & 0.049
 & 0.331 & 1.000 & 0.674
 & 0.423 & 1.000 & 0.949   \\
    $(200,50,25)$       & 0.048 &  0.051 & 0.050
& 0.949 & 1.000 & 1.000
&  0.997 & 1.000 & 1.000   \\
      $(200, 120, 60) $   & 0.048 &  0.045 & 0.044
&  0.395 & 1.000 & 0.889
& 0.454 & 1.000 & 0.994   \\
    $(200, 120, 90) $ &  0.059 & 0.053 & 0.056
& 0.420 & 1.000 & 0.857
&  0.497 & 1.000 & 0.997\\ 
    $(250,100,75)$ &
    0.048 & 0.061 & 0.050
& 0.810 & 1.000 & 0.997
& 0.907 & 1.000 & 1.000 \\
    $(250,150,100)$ &  0.053 & 0.052 & 0.053
 & 0.410 &  1.000 & 0.936
& 0.482 & 1.000 & 0.999 \\ 
    $(250,200,100)$ & 
   0.040 & 0.046 & 0.045
&  0.251 & 1.000 & 0.665
&  0.270 & 1.000 & 0.890 \\
\end{tabular}
\\[5ex]
     \caption{\label{table2} Empirical rejection rates of the tests $T_1, T_2$ and the LLRT under the null hypothesis \eqref{null} and the alternative model \eqref{model} for different values of $n,p_1,p_2, q=2$, $\mathbf{x} \sim \mathcal{N}(\mathbf{0}, \mathbf{I})$ and $2,000$ simulation runs. }
    \begin{tabular}{c ccc ccc ccc}
 & & $\delta = 0$ & & & $\delta = 0.1$ & & &  $\delta = 0.2$ & 
\\
\cline{2-10}
   $(n, p, p_1)$ & $T_1$ & $T_2$ & LLRT 
   & $T_1$ & $T_2$ & LLRT
   & $T_1$ & $T_2$ & LLRT \\
   $(150, 80, 70)  $ & 0.048 & 0.042 & 0.051
 & 0.086 & 0.103 & 0.092
 & 0.277 & 0.457 & 0.290 \\
    $(150,80,40)$ & 0.054 & 0.054 & 0.059
 & 0.188 & 0.287 & 0.183
& 0.915 & 0.990 & 0.895 \\
    $(150,100,80) $     & 0.057 & 0.060 & 0.049
&  0.090 & 0.122 & 0.093
& 0.356 & 0.642  & 0.349  \\
    $(200,50,25)$       & 0.048 &  0.051 & 0.050
 & 0.334 & 0.407 & 0.340
& 0.998 & 0.100 & 0.999   \\
      $(200, 120, 60) $   &  0.048 &  0.045 & 0.044
 & 0.251 & 0.397 & 0.238
& 0.981 & 0.100 & 0.976  \\
    $(200, 120, 90) $ &  0.059 & 0.053 & 0.056
 & 0.140 & 0.204 & 0.147
& 0.682 & 0.923 & 0.671 \\ 
    $(250,100,75)$ &
     0.048 & 0.061 & 0.050
&  0.181 & 0.259 & 0.180
& 0.921 & 0.981 & 0.915\\
    $(250,150,100)$ &  
    0.053 & 0.052 & 0.053
& 0.212 & 0.350 & 0.203
 &  0.947 & 0.100 & 0.937 \\ 
    $(250,200,100)$ & 
    0.040 & 0.046 & 0.045
 & 0.281 & 0.543 & 0.243
&  0.993 & 1.000 & 0.969 \\
\end{tabular}
    \end{table}

\subsection{The almost-singular case}

 \begin{figure}[!ht]
     \centering
  \includegraphics[width=0.3\columnwidth]{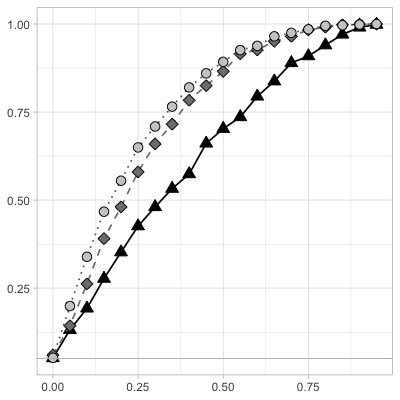}  \hspace{0.5cm}
   \includegraphics[width=0.3\columnwidth]{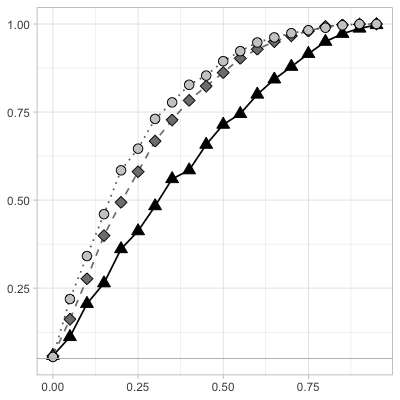}   \hspace{0.5cm}
    \includegraphics[width=0.3\columnwidth]{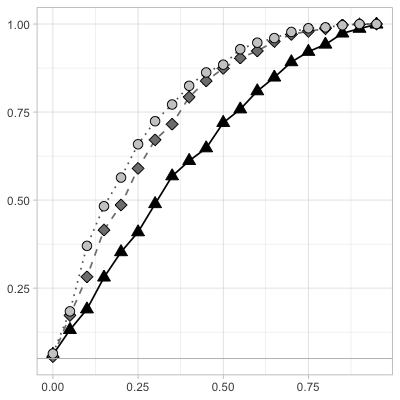}   
    
\vspace{0.5cm}
    
   \includegraphics[width=0.3\columnwidth]{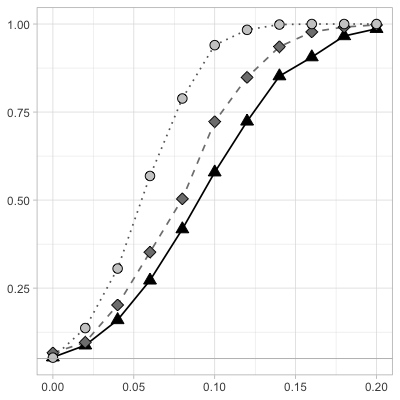} \hspace{0.5cm}
  \includegraphics[width=0.3\columnwidth]{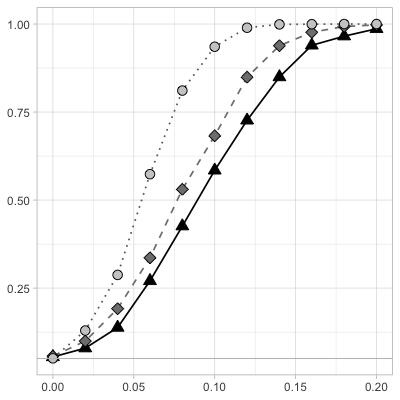}   \hspace{0.5cm}
 \includegraphics[width=0.3\columnwidth]{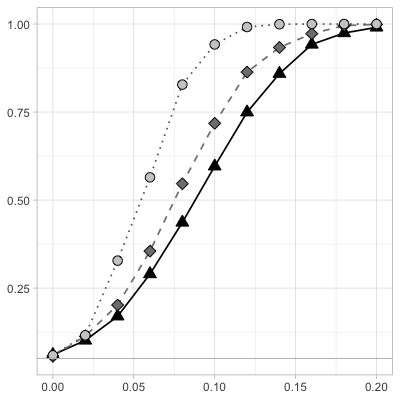}  
   
    \caption{ Rejection rates for the log-likelihood ratio test \eqref{test} for Scenario (i) (first row, defined at the beginning of Section \ref{sec_sim}) and Scenario (ii) (second row) under  \eqref{alternative} based on $2,000$ simulation runs. First column: standard normally distributed data, second column: standardized t-distributed data with $15$ degrees of freedom, third column: centered exponentially distributed data with parameter $1$. The triangle indicates $n=100, p=98$, the square $n=120, p=118$ and the circle $n=180, p=178$ for Scenario (i). For Scenario (ii), the corresponding parameters are $n=100, p=99$ (triangle), $n=120, p=117$ (square) and $n=180, p=177$ (circle), respectively.   
 The vertical gray line in each figure defines the nominal level $\alpha = 5 \%$.
       }\label{fig_power_sing}
       
\end{figure} 
\ND{
In Theorem \ref{thm}, we assume that the dimension-to-sample-size ratio $p/n$ is uniformly bounded away from $1$, which ensures that the variance $\sigma_n^2$ remains bounded. In this subsection, we will demonstrate that the test proposed in \eqref{test} can admit a desirable performance, especially under the null hypothesis \eqref{null}, even if the variance explodes. For this purpose, we consider Scenarios (i) and (ii) in cases where the dimension $p$ is close to the sample size $n$. The two scenarios determining the group structure are defined at the beginning of Section \ref{sec_sim}. 
In Fig. \ref{fig_power_sing}, we display the empirical rejection rates of the test \eqref{test} for different choices of $\delta$, $n$, $p$, $q$, $p_i$ and different distributions for the random vector $\bfx$. The nominal level $\alpha = 5\%$ is approximated well under the null hypothesis ($\delta =0$) in all cases under consideration. Comparing Fig. \ref{fig_power_sing} to Fig. \ref{fig_power}, the increase of the power curve is a little flatter if the dimension $p$ is close to the sample size $n$ for Scenario (ii) and significantly flatter for Scenario (i). This indicates that a central limit theorem for the log-likelihood test statistic under the null hypothesis might still hold true if $p/n \to 1$ as $n\to \infty$ but, unsurprisingly, at the cost of a lower speed of convergence as the variance is no longer bounded. Under the alternative for Scenario (i), the test does not yield a useful decision criterion. 
However, we still observe a reasonable behavior of the test under $H_0$ for both Scenarios (i) and (ii) as well as under the alternative for Scenario (ii).
 }

\subsection{Analysis of the 16 Personality Factor Questionnaire} 

To further illustrate the usefulness of our method, we analyze the 16 Personality Factor Questionnaire (16PF) data set. 
The questionnaire goes back to \cite{cattell1957personality, cattell1973} and provides a comprehensive measure of normal-range personality characteristics \citep[see also][]{cattell2008sixteen}. The 16 proposed factors include warmth, reasoning, emotional stability, dominance, liveliness, rule-consciousness, social boldness, sensitivity, vigilance, abstractedness, privateness, apprehension, openness to change, self-reliance, perfectionism, tension. In psychology, there is a controversy about the appropriate number of personality factors in connection with 16PF \citep[e.g., see][]{cattell1986number, mershon1988number}. Nowadays, most psychologists can agree on a number of 5 factors, the so-called Big Five: openness, conscientiousness, extraversion, agreeableness, neuroticism  \citep[see, e.g.][]{barrick1991big, de2000big, roccas2002big, gosling2003very}. Nevertheless, the 16PF is widely applied in various contexts where an in-depth personality analysis is needed. For example, it can be used in order to construct effective self-development goals and career plans adjusted to a worker's  individual strengths and limitations \citep[see, among many others,][]{carson1998integration, krug199016, watterson2002}. The relevance of 16PF is emphasized by the fact that it had been adapted to over 35 languages worldwide.

The data set "PF16" under consideration (available at \url{https://openpsychometrics.org/_rawdata/}) includes the ratings of $49159$ persons concerning $163$ statements about their personality on the following scale of accuracy: (1) disagree, (2) slightly disagree, (3) neither agree nor disagree, (4) slightly agree, (5) agree. 

For the following analysis, we centered the data with the sample mean. Since we will perform the test only on subsamples of size $n << 49159$, this will not cause a mean shift in the CLT given in Theorem \ref{thm} as discussed in detail at the end of this subsection. 
The number of groups is $q=16$ with sizes $p_i = 10$ for $i\in \{1, \ldots, 16\} \setminus \{ 2 \} $ and $p_2 = 13,$ so that the total dimension indeed sums up to $p=163$.
In order to quantify how far the sample covariance matrix $\hat\bfSigma$ deviates from a block diagonal structure, we propose the following quantity
\begin{align*}
    \tau(\hat \bfSigma) = \frac{ || \hat\bfSigma - \hat\bfSigma_{\textnormal{block}}  || }{|| \hat\bfSigma || } \geq 0. 
\end{align*}
Here, 
\begin{align*}
    \hat\bfSigma_{\textnormal{block}} = \begin{pmatrix}
    \hat\bfSigma_{11} & \mathbf{0} & \mathbf{0} & \ldots & \mathbf{0} \\
    \mathbf{0} & \hat\bfSigma_{22} & \mathbf{0} & \ldots & \mathbf{0} \\ 
    \vdots & \vdots & \vdots & \ddots & \vdots \\ 
    \mathbf{0} & \mathbf{0} & \mathbf{0} & \ldots & \hat\bfSigma_{16,16}
    \end{pmatrix}
\end{align*}
is the block-diagonal version of $\hat\bfSigma$
and $|| \cdot ||$ denotes an appropriate matrix norm.  Note that $\tau(\hat \bfSigma)$ equals $0$ if and only if $\hat \bfSigma = \hat\bfSigma_{\textnormal{block}}$ and $\hat\bfSigma \neq \mathbf{0}$. For the $1$-norm or the Frobenius norm, we have $\tau(\hat \bfSigma) \in [0,1]$ and $\tau(\hat \bfSigma)$ equals $1$ if and only if $\hat\bfSigma_{ii} = \mathbf{0}$ for $1 \leq i \leq 16.$ ($\mathbf{0}$ denotes always a matrix of appropriate dimension filled with zeros.) 
In Table \ref{table3}, we display the values of $\tau(\hat\bfSigma)$ for different matrix norms. 
\begin{table}[h]
\begin{center}
\caption{Values of $\tau(\hat\bfSigma)$ for different matrix norms.} \label{table3}
\begin{tabular}{cccc}
    Norm & spectral & Frobenius & 1-norm   \\
    $\tau(\hat \bfSigma)$ & $0.7962393$    & $0.8005626$  & $0.825113$  
\end{tabular}
\end{center}
\end{table} 
Since the values of $\tau (\hat\bfSigma)$ are not close to $0$, assuming a block diagonal structure seems not reasonable and the null hypothesis of blockdiagonality should be rejected. 
In order to demonstrate the usefulness of our approach for high-dimensional data sets, we proceed with a data reduction. We interpret the answers to the questionnaire as i.i.d. observations from a random vector of dimension $p=163$. Then we choose a sample of persons of size $n<<49159$ uniformly at random and iterate this procedure. 
The test in \eqref{test} rejects the null hypothesis \eqref{null} of blockdiagonality for $p, q, p_i$ as given above and all values of the size $n$ of the subsample under consideration (e.g., $n\in \{ 170, 200, 250\}$), that is, the empirical rejection rate equals $1$. This means that the LLRT detects the dependency between the 16 factors indicated by the values of $\tau(\hat\bfSigma)$. 

Motivated by the Big Five, we choose five subgroups corresponding to the factors openness to change, perfectionism, privateness, warmth and apprehension, and analyze the corresponding sample covariance matrix $\hat\bfSigma_{\textnormal{sub}}.$ The “measure” $\tau(\hat\bfSigma_{\textnormal{sub}})$ for blockdiagonality of $\hat\bfSigma_{\textnormal{sub}}$ can be defined similarly to $\tau (\hat\bfSigma)$ and is displayed in Table \ref{table4}.
\begin{table}[h]
\caption{Values of $\tau(\hat\bfSigma_{\textnormal{sub}})$ for different matrix norms.} \label{table4}
\begin{center} 
\begin{tabular}{cccc}
    Norm & spectral & Frobenius & 1-norm   \\
     $\tau(\hat\bfSigma_{\textnormal{sub}})$&  $0.5270513$  & $0.5003968$  & $0.4909513$
\end{tabular}
\end{center} 
\end{table}
The values of $\tau(\hat\bfSigma_{\textnormal{sub}})$ displayed above differ significantly from $0$ but less extremely than those of $\tau (\hat\bfSigma)$. Still, we expect that a reasonable test rejects the null hypothesis at a high rate when iterated over randomly chosen subsamples. 
In the Table \ref{table5}, the empirical rejection rate is displayed for the test given in \eqref{test} applied to  $500$ subsamples of size $n\in \{ 51, 55, 60, 65, 70, 75\} $ chosen at random and dimension $p=50$, $q=5$, $p_i=10$ for $1 \leq i \leq 5$. The choice of the subgroups is motivated by the Big Five, as explained above. 
\begin{table}[h]
\caption{Empirical rejection rate for the test \eqref{test} applied to  $500$ subsamples of size $n\in \{ 51, 55, 60, 65, 70, 75\} $ chosen at random and dimension $p=50$, $q=5$, $p_i=10$ for $1 \leq i \leq 5$.}
\label{table5}
\begin{center} 
\begin{tabular}{c p{2cm} p{2cm} p{2cm} p{2cm} p{2cm} p{2cm}}
    $n$ & 51 & 55 &  60 & 65 & 70 & 75 \\
    rejection rate & 0.774 & 0.92 & 0.976 &0.992 & 0.966 & 1   \\
\end{tabular}
\end{center} 
\end{table} 
Desirably, as the values of $\tau(\hat\bfSigma)$ and  $\tau(\hat\bfSigma_{\textnormal{sub}})$ are not close to $0$, we observe that the null hypothesis is rejected for an overwhelming number of cases under consideration. In this section, we have seen that the LLRT does not only perform well for simulated data in various scenarios, but it can also be used in order to detect dependent factors appearing in real data sets. 

Moreover, our statistical analysis reflects the controversy in psychology about the appropriate number of factors describing a human's personality. It confirms that the factors given in 16PF are not independent, even when choosing a smaller group of factors which can be associated with the Big Five. This observation sheds light on a fundamental difference between these two approaches: 
Although personality traits are thought to be dependent, the Big Five in contrast to 16PF are modelled as independent factors received through the factor analysis using orthogonal rotations, which simplifies the statistical analysis and the psychological interpretation. 
This may explain why the Big Five are very popular as a universal model in personality research, while the 16PF is widely applied in various everyday scenarios, e.g., by mental health professionals and career coaches. 

Concluding this data example, we would like to comment on two technical issues relating to Theorem \ref{thm}. First, the assumption that the data follows a continuous distribution is obviously violated.  However, as the number $q$ of groups is relatively small in comparison to the considered sample sizes $n$, we can interpret $q$ asymptotically as a fixed (or negligible) quantity in comparison to $n.$
Then, the continuity assumption in Theorem \ref{thm} is not necessary (see Remark \ref{rem_substitution}). 
Second, the data are assumed to be centered, and in general, we expect a mean shift in Theorem \ref{thm} when centering with the sample mean. In our application, we center with the full sample mean, while the statistical analysis is restricted to subsamples of much smaller size $n << 49159$. Thus, the full sample mean can be interpreted as the true population mean and does not cause a shift in the mean structure.

\section{Proofs} \label{sec_proofs}
In this section, we provide the proof of Theorem \ref{thm} and present the necessary auxiliary results. We conclude with the proof of Theorem \ref{thm_eq_cov}.
In the following, we make use of the notation $a\lesssim b$ which means that $a$ is less than or equal to $b$ up to a positive constant, that is, there exists some $C>0$ independent of $n\in\N$ such that $a \leq C b$. 
\begin{proof}[\textbf{\upshape Proof of Theorem \ref{thm}. }] 
Note that, under the null hypothesis \eqref{null}, we have
\begin{align} \label{a2}
	V_n = \frac{ |\bfSigma| | \hat{\mathbf{I}} | } { \prod\limits_{i=1}^q \lb  |\bfSigma_{ii}| | \hat{\mathbf{I}}_{ii} | \rb }
	= \frac{  | \hat{\mathbf{I}} | } { \prod\limits_{i=1}^q | \hat{\mathbf{I}}_{ii} | },
\end{align}
	where
	\begin{align*}
		\hat{\mathbf{I}} = & \frac{1}{n} \bfX_n \bfX_n^\top, 
		\hat{\mathbf{I}}_{ii} =  \frac{1}{n} \bfX_{n,i} \bfX_{n,i}^\top, 
		\bfX_n =  (\bfx_1, \ldots \bfx_n) = (\bfb_1, \ldots, \bfb_p)^\top , \\
		\bfX_{n,i} = & (\bfb_{p_{i-1}^\star+1}, \ldots, \bfb_{p_i^\star})^\top, 
		p_i^\star =  \sum\limits_{j=1}^i p_j ,
	\end{align*}
	for $1 \leq i \leq q$, where we set $ p_0^\star =0$.
In order to establish a more handy representation for determinants of the sample covariance matrix, we proceed with a QR decomposition of $\bfX_n^\top$ and $\bfX_{n,i}^\top$ as explained in detail in Section 2 of \cite{wang2018} and get
	\begin{align*}
		\hat{\mathbf{I}} = &  \frac{1}{n} \prod\limits_{i=1}^p \bfb_i^\top \bfP (i-1) \bfb_i, ~ 
		\hat{\mathbf{I}}_{ii} =  \frac{1}{n} \prod\limits_{j=p_{i-1}^\star+1}^{p_i^\star} \bfb_j^\top \bfP (p_{i-1}^\star+1; j-1) \bfb_j , ~ 1 \leq i \leq q,
	\end{align*}
	where
	\begin{align}
		 & \bfP(p_{i-1}^\star+1 ; j-1)  
		  =  \mathbf{I}   -\bfX_{n} (p_{i-1}^\star+1 ; j-1)^\top
	\lb \bfX_n (p_{i-1}^\star+1 ; j-1) \bfX_n (p_{i-1}^\star+1 ; j-1)^\top \rb\inv \bfX_n (p_{i-1}^\star+1 ; j-1), \label{def_projection_matrix}
	\end{align}
	and 
	\begin{align*}
	\bfP(j) =& \mathbf{P}(1;j), 
	\end{align*}
	denote the projection matrices on the orthogonal complements of span$(\bfb_{p_{i-1}^\star+1}, \ldots \bfb_{j-1} )$ and span$(\bfb_1, \ldots \bfb_{j} )$, respectively. Here, we denote
	\begin{align*}
			\bfX_n (i;j) = & (\bfb_i, \ldots, \bfb_j)^\top, ~ 1 \leq i \leq j\leq p
	\end{align*}
	and $\bfP(0) = \mathbf{I} = \bfP(i;j)$ for $j<i$. 
	This implies
	\begin{align*}
		\log V_n = &
		\sum\limits_{i=1}^p \log \lb \bfb_i^\top \bfP(i-1) \bfb_i \rb 
		- \sum\limits_{i=1}^q \sum\limits_{j=p_{i-1}^\star+1}^{p_i^\star} \log \lb \bfb_j^\top \bfP(p_{i-1}^\star+1;j-1) \bfb_j \rb \\
		= & \sum\limits_{i=p_1 + 1}^p \log \lb \bfb_i^\top \bfP(i-1) \bfb_i \rb 
		- \sum\limits_{i=2}^q \sum\limits_{j=p_{i-1}^\star+1}^{p_i^\star} \log \lb \bfb_j^\top \bfP(p_{i-1}^\star+1;j-1) \bfb_j \rb,
	\end{align*}
	where we used $\bfP(i - 1) = \bfP(1;i - 1)$ for $1\leq i \leq p_1$. 
	In the following, we will make use of Stirling's formula
	\begin{align*}
		\log n! = n \log n - n + \frac{1}{2 } \log ( 2 \pi n) + \frac{1}{12 n}
		+ \mathcal{O} \lb n^{-3} \rb, ~ n\to\infty. 
	\end{align*}
	As a preparation, we note that
	\begin{align}
		& \sum\limits_{i=p_1+1}^p  \log ( n - i + 1) 
		- \sum\limits_{i=2}^q \sum\limits_{j=p_{i-1}^\star+1}^{p_i^\star} \log  ( n - j + 1 + p_{i-1}^\star )  
		=  \log \frac{( n - p_1)!}{(n - p)!} - \sum\limits_{i=2}^q \log \frac{n!}{(n - p_i)!} \nonumber \\
		= & \sum\limits_{i=2}^q \lb n - p_i + \frac{1}{2} \rb \log \lb 1 - \frac{p_i}{n} \rb 
		- \lb n - p + \frac{1}{2} \rb \log \lb 1 - \frac{p}{n} \rb 
		- \sum\limits_{i=2}^q \lb \frac{1}{12 n } - \frac{1}{12 ( n - p_i) }\rb
		+ o(1) \nonumber \\
		= & \mu_n + \frac{ \sigma_n^2}{2} 
		- \frac{1}{12} \sum_{i=2}^q \frac{p_i}{n (n - p_i) }
		+   o(1)
		=  \mu_n + \frac{ \sigma_n^2}{2}  + o(1), \label{cal_mu}
	\end{align}
	where we used the fact 
	\begin{align} \label{n-p_i_infty}
	\min_{1 \leq i \leq q} (n - p_i) \to\infty,
	\end{align}
	which is a consequence of our assumptions. 
	Defining for $p_1 + 1 \leq i \leq p, 2 \leq j \leq q$
	\begin{align*}
		X_i  =  & \frac{ \bfb_i^\top \bfP(i-1) \bfb_i - (n - i + 1 ) }{ n - i +1}, ~
		X_{j,i}  =  \frac{ \bfb_i^\top \bfP(p_{j-1}^\star+1;i-1) \bfb_i - (n - i + 1 + p_{j-1}^\star ) }{n - i + 1 + p_{j-1}^\star },
		\\ 
		Y_i = & \log ( 1 + X_i) - \lb X_i - \frac{X_i^2}{2} \rb, ~ 
		Y_{j,i} =  \log ( 1 + X_{j,i}) - \lb X_{j,i} - \frac{X_{j,i}^2}{2} \rb,
	\end{align*}
	we decompose, using \eqref{cal_mu},
	\begin{align*}
		& \log V_n - \mu_n 
		 = \sum\limits_{i=p_1+1}^p X_i 
		-  \sum\limits_{j=2}^q \sum\limits_{i=p_{j-1}^\star+1}^{p_j^\star} X_{j,i} 
		-  \lb  \sum\limits_{i=p_1+1}^p \frac{X_i^2}{2} 
		- \sum\limits_{j=2}^q \sum\limits_{i=p_{j-1}^\star+1}^{p_j^\star} \frac{X_{j,i}^2}{2} 
		- \frac{\sigma_n^2}{2} 
		\rb 
		 + \sum\limits_{i=p_1+1}^p Y_i 
		- \sum\limits_{j=2}^q \sum\limits_{i=p_{j-1}^\star+1}^{p_j^\star} Y_{j,i} 
		+ o_{\textnormal{Pr}}(1).
	\end{align*}
	The assertion of Theorem \ref{thm} is then implied by Lemmas \ref{lem_conv_x},  \ref{lem_conv_x_pow2} and \ref{lem_conv_log_term}.
	\end{proof}
	\subsection*{Auxiliary results for the proof of Theorem \ref{thm} } 
	This section contains the auxiliary results needed for the proof of Theorem \ref{thm} and its proofs. 
	\begin{lemma} \label{lem_conv_x}
		Under the assumptions of Theorem \ref{thm}, it holds
		\begin{align*}
			\frac{\sum\limits_{i=p_1+1}^p X_i 
		- \sum\limits_{j=2}^q \sum\limits_{i=p_{j-1}^\star+1}^{p_j^\star} X_{j,i} }{\sigma_n} \cond \mathcal{N}(0,1), ~ n\to\infty. 
		\end{align*}
	\end{lemma} 
	\begin{proof}[\textbf{\upshape Proof of Lemma \ref{lem_conv_x}. }] 
	Let $\mathcal{F}_i = \sigma ( \{ \bfb_1, \ldots, \bfb_i \} )$ denote the $\sigma$-field generated by $\bfb_1, \ldots, \bfb_i$ for $1 \leq i \leq p$. We write 
	\begin{align*}
			\sum\limits_{j=p_1+1}^p X_j 
		- \sum\limits_{i=2}^q \sum\limits_{j=p_{i-1}^\star+1}^{p_i^\star} X_{i,j} 
		= & \sum\limits_{i=2}^q \sum\limits_{j=p_{i-1}^\star+1}^{p_i^\star} \lb X_j - X_{i,j} \rb  
		= \sum\limits_{i=p_1+1}^p Z_i , 
	\end{align*}
	where
	\begin{align*}
		Z_i = & X_i - X_{g(i),i} , ~ p_1 + 1 \leq i \leq p,
	\end{align*}
	where $g(i) = k$ if $\bfx^{(i)}$ belongs to the $k$th group, that is, if $p_{k-1}+1 \leq i \leq p_k$. 
	We observe that 
	$
		 \E \left[ Z_i | \mathcal{F}_{i-1} \right]  
		=  0		
	$
	and that $Z_i$ is measurable with respect to $\mathcal{F}_i$ for $p_1 +1 \leq i \leq p$. Thus, we conclude that $(Z_i)_{p_1+1 \leq i \leq p}$ forms a martingale difference scheme with respect to the filtration $(\mathcal{F}_i)_{p_1 + 1 \leq i \leq p}$ scheme for every $n\in\N$. We aim to apply a central limit theorem for this dependency structure. 
	In order to calculate the limiting variance, we write
	\begin{align*}
		Z_i  & =  \bfb_i^\top 
		\frac{( n - i + 1 + p_{g(i) - 1}^\star) \bfP(i-1) - (n - i + 1) \bfP( p_{g(i)-1}^\star+1 ;i-1)}{ (n - i + 1) ( n - i + 1 + p_{g(i) - 1}^\star) }
		\bfb_i \\
		& \quad - \tr \lb \frac{( n - i + 1 + p_{g(i) - 1}^\star) \bfP(i-1) - (n - i + 1) \bfP( p_{g(i)-1}^\star+1 ;i-1)}{ (n - i + 1) ( n - i + 1 + p_{g(i) - 1}^\star) } \rb 		
		, ~ p_1 + 1 \leq i \leq p,
	\end{align*}
	and use the fact
	\begin{align} \label{var_quad_form}
		\E \lb \bfb_i^\top \mathbf{A} \bfb_i - \tr \mathbf{A} \rb^2 = 
		2 \tr \mathbf{A}^2 + (\nu_4 - 3 )  \tr \lb \mathbf{A}^{\odot 2} \rb, ~ 1 \leq i \leq p,
	\end{align}
	for any non-random matrix symmetric $\mathbf{A} \in\R^{n\times n}$, where $\nu_4 = \E [ b_{11}^4]$. Here, $\odot$ denotes the Hadamard product of matrices (entry-wise multiplication) and we use the notation $$ \mathbf{A} \odot \mathbf{A} = \mathbf{A}^{\odot 2}. $$
 
	Consequently, we observe
	\begin{align*}
		\E [ Z_i^2 | \mathcal{F}_{i-1}] 
		= & \E [ ( X_i - X_{g(i),i} )^2 | \mathcal{F}_{i-1} ] 
		=  2 \lb  \frac{ \tr \bfP(i-1)^2 }{ ( n - i + 1) ^2 } + \frac{ \tr \bfP( p_{g(i)-1}^\star+1 ;i-1)^2 } { ( n - i + 1 + p_{g(i) - 1}^\star)^2 }  - 2 \frac{\tr \lb \bfP( p_{g(i)-1}^\star+1 ;i-1) \bfP(i-1) \rb  }{ ( n - i + 1 + p_{g(i) - 1}^\star) ( n - i + 1) }\rb \nonumber \\
		& + (\nu_ 4 - 3 ) \Bigg(  \frac{ \tr \lb \bfP(i-1)^{ \odot 2}   \rb  }{ ( n - i + 1) ^2 } 
		+  \frac{ \tr \lb  \bfP( p_{g(i)-1}^\star+1 ;i-1) ^{\odot 2} \rb  } { ( n - i + 1 + p_{g(i) - 1}^\star)^2 }  
		 - 2 \frac{ \tr \lb  \bfP( p_{g(i)-1}^\star+1 ;i-1) \odot \bfP( i-1) \rb  } { ( n - i + 1 + p_{g(i) - 1}^\star) ( n - i + 1 ) } 
		\Bigg) \nonumber  \\
		= & \sigma_{n,1,i}^2 + (\nu_4 - 3) \sigma_{n,2,i}^2, 
	\end{align*}
	where
	\begin{align} 
		\sigma_{n,1,i}^2 & =  2 \lb  \frac{ 1 }{  n - i + 1} -  \frac{1 } { n - i + 1 + p_{g(i) - 1}^\star }  \rb, \nonumber \\
		\sigma_{n,2,i}^2 & =  \frac{ \tr \lb \bfP(i-1)^{\odot2} \rb  }{ ( n - i + 1) ^2 } 
		+  \frac{ \tr \lb  \bfP( p_{g(i)-1}^\star+1 ;i-1)^{ \odot 2} \rb  } { ( n - i + 1 + p_{g(i) - 1}^\star)^2 } 
		- 2 \frac{ \tr \lb  \bfP( p_{g(i)-1}^\star+1 ;i-1) \odot \bfP( i-1) \rb  } { ( n - i + 1 + p_{g(i) - 1}^\star) ( n - i + 1 ) } . \label{def_sigma_2}
	\end{align}
	Note that $\sum_{i=p_1 +1}^p \sigma_{n,2,i}^2 = o_{\textnormal{Pr} }(1)$ by Lemma \ref{lem_conv_sigma2i}. 
	For the term $\sigma_{n,1,i}^2$, we used that $\bfP(i-1) \bfP( p_{g(i)-1}^\star+1 ;i-1)  = \bfP (i-1)$. 
	Thus, we have for this term
	\begin{align}
		\sum\limits_{i= p_1 + 1}^p \sigma_{n,1,i}^2 
		= & 2 \lb \log (n - p_1 ) - \log ( n - p ) 
		-   \sum\limits_{i=2}^q \sum\limits_{j=p_{i-1}^\star+1}^{p_i^\star} \frac{1}{n - j + 1 + p_{i-1}^\star}  \rb  
		+ o(1) \nonumber \\
		=  &  2 \Bigg\{ \log (n - p_1 ) - \log ( n - p ) 
		-   \sum\limits_{i=2}^q \left\{   \log (n) - \log ( n - p_i^\star + p_{i-1}^\star)\right\}   
		 - \sum\limits_{i=2}^q \lb \frac{1}{2 n} - \frac{1}{2 ( n - p_i) } \rb \Bigg\}
		+ o(1) \nonumber \\
		=  &  2 \lb \log \lb 1 - \frac{p_1}{n} \rb  - \log \lb 1 - \frac{p}{n} \rb  
		+   \sum\limits_{i=2}^q   \log \lb 1 - \frac{p_i }{n} \rb  
		+ \frac{1}{2} \sum_{i=2}^q \frac{p_i}{n (n - p_i) }\rb  
		+ o(1) \nonumber \\
		= &   2 \lb  \sum\limits_{i=1}^q   \log \lb 1 - \frac{p_i }{n} \rb   - \log \lb 1 - \frac{p}{n} \rb \rb 
		+ o(1)
		= \sigma_n^2 + o(1), \label{a3}
	\end{align}
	where we used \eqref{n-p_i_infty}, $\sum_{i=2}^q p_i \leq p$ and the expansion for the partial sums of the harmonic series 
	\begin{align}
	\label{harmonic_series}
	\sum\limits_{k=1}^n \frac{1}{k}
	= \log n + \gamma + \frac{1}{2n} +  \mathcal{O}\lb n^{-2} \rb, ~ n\to\infty.
	\end{align}
	Here, $\gamma$ denotes the Euler-Mascheroni constant. 
	Note that the term in \eqref{a3} is bounded away from zero for all $n\in\N$. More precisely, we have applying inequality (33) of \cite{qi_et_al_2019}
	\begin{align} \label{bound_sigma}
		\sigma_n^2 \geq  2 \left\{ \sum\limits_{i=1}^q \log \lb 1 - \frac{p_i}{n} \rb 
				- \log \lb 1 - \frac{p}{n} \rb \right\} \geq  - 2\lb  \log \lb 1 - \frac{p}{n} \rb + \frac{p}{n} \rb ( 1-\eta)  > 0
	\end{align}
	uniformly over $n\in\N$ (recall that $\inf_{n\in\N} p/n > 0$ and $\max_{1 \leq i \leq q} p_i \leq \eta p$). 
	These considerations imply 
	\begin{align} \label{conv_var}
		 \sum\limits_{i=p_1 +1}^p \E \left[ \frac{Z_i^2}{\sigma_n^2} \Big| \mathcal{F}_{i-1} \right] \conp 1, ~n\to\infty. 
	\end{align}
	\\ Let $\varepsilon > 0$.  Then, using Lemma B.26 in \cite{bai2004} and recalling that $\E | x_{11}|^{4+\delta} <\infty$, we get 
	\begin{align}
		& \sum\limits_{i= p_1 +1}^p \E [ Z_{i}^2 I\{|Z_i| > \varepsilon \} ] 
		\leq   \frac{1}{\varepsilon^{\frac{\delta}{2}}}\sum\limits_{i= p_1 +1}^p \E [ |Z_{i}|^{2+\frac{\delta}{2}} ]
		\lesssim  \sum\limits_{i= p_1 +1}^p \frac{1}{\lb n - i  + 1\rb^{2+\frac{\delta}{2}}}
		\E \left| \bfb_i^\top \bfP(i-1) \bfb_i - (n - i + 1 ) \right|^{2+\frac{\delta}{2}} 
		\nonumber \\
		& + \sum\limits_{i=2}^q \sum\limits_{j=p_{i-1}^\star+1}^{p_i^\star}  \frac{1}{\lb n - j + 1 + p_{i-1}^\star\rb^{2+\frac{\delta}{2}}}
		\E \left| \bfb_j^\top \bfP(p_{i-1}^\star+1;i-1) \bfb_j - (n - j + 1 + p_{i-1}^\star ) \right|^{2+\frac{\delta}{2}} 
		\nonumber \\
		\lesssim & \sum\limits_{i= p_1 +1}^p \frac{1}{\lb n - i  + 1\rb^{1+\frac{\delta}{4}}}
		+ \sum\limits_{i=2}^q \sum\limits_{j=p_{i-1}^\star+1}^{p_i^\star}  \frac{1}{\lb n - j + 1 + p_{i-1}^\star\rb^{1+\frac{\delta}{4}}}  
		= o(1), ~n\to\infty. 
		\nonumber
	\end{align}
	Using \eqref{bound_sigma} and the fact that $(\sigma_n^2)_{n\in\N}$ is bounded, we see that $(Z_i/ \sigma_n )_{p_1 + 1 \leq i \leq p}$ satisfies the following Lindeberg condition for all $\varepsilon > 0$:
	\begin{align}
		\sum\limits_{i= p_1 +1}^p \E \left[ \frac{ Z_{i}^2 }{\sigma_n^2} I\{ \left| Z_i / \sigma_n \right| > \varepsilon \} \right] 
		= o(1), ~n\to\infty. \label{lindeberg}
	\end{align}
	Since \eqref{conv_var} and \eqref{lindeberg} hold true, we may apply a CLT for martingale difference schemes (e.g., see Corollary 3.1 in \cite{hall_heyde}) and the proof of Lemma \ref{lem_conv_x} concludes. 	
\end{proof}

\begin{lemma} \label{lem_conv_x_pow2}
	 Under the assumptions of Theorem \ref{thm}, it holds
	 \begin{align*}
	 	\frac{ \sum\limits_{i=p_1+1}^p \frac{X_i^2}{2} 
		- \sum\limits_{j=2}^q \sum\limits_{i=p_{j-1}^\star+1}^{p_j^\star} \frac{X_{j,i}^2}{2} 
		- \frac{\sigma_n^2}{2} } {\sigma_n} \conp 0, ~  n \to\infty.
	 \end{align*}
	\end{lemma} 	
		\begin{proof}[\textbf{\upshape Proof of Lemma \ref{lem_conv_x_pow2}.}]  
	Note that for each $n\in\N$, both $(X_i)_{p_1 + 1 \leq i \leq p}$ and 
	$ ( X_{g(i),i} )_{p_1 + 1 \leq i \leq p} $
	form a martingale difference scheme with respect to the filtration  $(\mathcal{F}_i)_{p_1 + 1 \leq i \leq p}$ defined previously. 
	We obtain from the proof of Lemma \ref{lem_conv_x} 
		\begin{align}
		& \sum\limits_{i=p_1+1}^p \E [ X_i^2 | \mathcal{F}_{i-1} ]  
		= \check \sigma_{n,1}^2 + o_{\PR}(1), ~ n\to\infty, \label{var_xi}
	\end{align}
	and
\begin{align}
		& \sum\limits_{i=2}^q \E \left[  X_{g(i),i}^2 \Bigg| \mathcal{F}_{i-1} \right] 
		= \sum\limits_{i=2}^q \sum\limits_{j=p_{i-1}^\star+1}^{p_i^\star} \E [ X_{i,j}^2 | \mathcal{F}_{i-1} ] 
		=  \check \sigma_{n,2}^2 + o_{\PR} (1), ~ n\to\infty, \label{var_xij}
	\end{align}
	where we define 
	\begin{align*}
	\check\sigma_{n,1}^2 & = - 2 \log \lb 1 - \frac{p}{n} \rb 
		+  ( \nu_4 - 3) \sum\limits_{i=p_1+1}^p  \frac{ \tr \lb \bfP(i-1)^{ \odot 2}   \rb  }{ ( n - i + 1) ^2 } ,\\
			\check\sigma_{n,2}^2 & =  -  2  \sum\limits_{i=1}^q   \log \lb 1 - \frac{p_i }{n} \rb 
		+  ( \nu_4 - 3) \sum\limits_{i=p_1+1}^p   \frac{ \tr \lb  \bfP( p_{g(i)-1}^\star+1 ;i-1) ^{\odot 2} \rb  } { ( n - i + 1 + p_{g(i) - 1}^\star)^2 } .
	\end{align*}
	Recalling that $0< \inf_{n\in\N} \min_{1 \leq i \leq q} ( p_i q)/n \leq \sup_{n\in\N} p/n < 1$ and using the inequality $\log(1+x) \leq x$ for $x> -1$, we note that
	\begin{align}
		0 < \inf_{n\in\N} \check \sigma_{n,1}^2 \leq \sup_{n\in\N} \check \sigma_{n,1}^2 < \infty , ~ 
		0< \inf_{n\in\N} \check \sigma_{n,2}^2 \leq \sup_{n\in\N} \check \sigma_{n,2}^2 <\infty .
		\label{bound_sigma_check}
	\end{align}
	Taking a closer look at the proof of \eqref{lindeberg}, we observe that both schemes satisfy the Lindeberg condition, that is, we have for $\varepsilon > 0$
	\begin{align}
	& \sum\limits_{i= p_1 +1}^p \E \left[ \frac{X_{i}^2 }{ \check \sigma_{n,1}^2}  I\{|X_i / \check \sigma_{n,1} | > \varepsilon \}  \right] 
	= o(1),
	\label{lindeberg_xi}  \\
	& \sum\limits_{i= p_1 +1}^p \E \left[ \frac{   X_{g(i),i}^2 }{\check \sigma_{n,2}^2}  I\{|X_{g(i),i} / \check \sigma_{n,2} | > \varepsilon \} \right] 
	= o(1) .
	\label{lindeberg_xij}
	\end{align}
	\ND{By Theorem 2.23 in \cite{hall_heyde}}, we see that \eqref{var_xi}, \eqref{var_xij}, \eqref{lindeberg_xi} and \eqref{lindeberg_xij} imply that the conditional variance can be approximated by the sum of squares, that is, 
	\begin{align*}
		& \sum\limits_{i=p_1+1}^p \frac{  X_i^2 - \E [ X_i^2 | \mathcal{F}_{i-1} ] }{\check\sigma_{n,1}^2} \conp 0, ~ 
		 \sum\limits_{i=p_1+1}^p \frac{ X_{g(i),i}^2 - \E \left[  X_{g(i),i}^2 | \mathcal{F}_{i-1} \right] }{\check\sigma_{n,2}^2} \conp 0, 
		~n\to\infty, 
	\end{align*}
	Combining these observations with \eqref{var_xi}, \eqref{var_xij} and \eqref{bound_sigma_check}, we get
	\begin{align*}
		& \sum\limits_{i=p_1+1}^p  X_i^2	- \check \sigma_{n,1}^2 \conp 0, ~
		\sum\limits_{i=p_1+1}^p  X_{g(i),i}^2	- \check \sigma_{n,2}^2 \conp 0.
	\end{align*}
	Using \eqref{bound_sigma} and $ \check \sigma_{n,1}^2 - \check \sigma_{n,2}^2 = \sigma_n^2 + o_{\PR}(1)$ by Lemma \ref{lem_conv_pii_sq}, the proof of Lemma \ref{lem_conv_x_pow2} concludes. 
\end{proof}

\begin{lemma} \label{lem_conv_log_term}
	Under the assumptions of Theorem \ref{thm}, it holds
		\begin{align}
		   & \sum\limits_{i=p_1+1}^p \frac{ Y_{i} }{\sigma_n} \conp 0, ~ \label{conv_log1}\\
			& \sum\limits_{j=2}^q \sum\limits_{i=p_{j-1}^\star+1}^{p_j^\star} \frac{Y_{j,i}}{\sigma_n} \conp 0, 
			\label{conv_log2}
		\end{align}
		as $n\to\infty$.
	\end{lemma} 	
		\begin{proof}[\textbf{\upshape Proof of Lemma  \ref{lem_conv_log_term}.} ]
	In the following, we will show that the convergence in \eqref{conv_log1} holds true. Then, the assertion \eqref{conv_log2} can be shown similarly. \\ 
	Let $0 < \varepsilon <1$. Then, we estimate for $1 + p_1 \leq i \leq p$ using Taylor's expansion
	\begin{align*}
		\E \left[ | Y_i | I \{ | X_i| \leq 1 - \varepsilon \} \right]
		\lesssim \E \left[ | X_i|^3 I \{ | X_i| \leq 1 - \varepsilon \} \right]
		\lesssim \E | X_i|^{2+\frac{\delta}{2}}.
	\end{align*}
	We also have
	\begin{align*}
		& \E \left[ | Y_i | I \{ | X_i| > 1 - \varepsilon \} \right] 
		\leq  
		\E \left[ | \log(1 + X_i) | I \{ | X_i| > 1 - \varepsilon \} \right]
		+ \E \left[ | X_i | I \{ | X_i| > 1 - \varepsilon \} \right]
		+ \E \left[  X_i^2 I \{ | X_i| > 1 - \varepsilon \} \right]  \\
		\lesssim &   \E \left[ | X_i | I \{ | X_i| > 1 - \varepsilon \} \right]
		+ \E \left[  X_i^2 I \{ | X_i| > 1 - \varepsilon \} \right]  
		\lesssim  \E | X_i|^{2+\frac{\delta}{2}},
	\end{align*}
	where we used the inequality $\log (1 +x ) \leq x $ for all $x > - 1$.
	These two estimates imply
	\begin{align*}
		\sum\limits_{i=p_1 + 1}^p \E | Y_i| \lesssim \sum\limits_{i=p_1 + 1}^p \E | X_i|^{2+\frac{\delta}{2}} =o(1), ~ n\to\infty,
	\end{align*}
	where we used Lemma B.26 in \cite{bai2004} as in the proof of \eqref{lindeberg}. Thus, we obtain
	\begin{align*}
		\sum\limits_{i=p_1 + 1}^p Y_i \conp 0, ~
		n \to\infty,
	\end{align*}
	which implies the assertion of Lemma \ref{lem_conv_log_term} recalling $\inf_{n\in\N} \sigma_n^2 > 0$. 
\end{proof}	
	
	\begin{lemma} \label{lem_conv_pii_sq}
		It holds, as $n\to\infty$, 
		\begin{align}
		 &  \sum\limits_{i=p_1+1}^p  \lb  \frac{ \tr \lb \bfP(i-1)^{ \odot 2}   \rb  }{ ( n - i + 1) ^2 } 
			   - \frac{1}{n} \rb \conp 0,  \label{z10} \\
			    &  \sum\limits_{i=p_1+1}^p  \lb  \frac{ \tr \lb  \bfP( p_{g(i)-1}^\star+1 ;i-1) ^{\odot 2} \rb  } { ( n - i + 1 + p_{g(i) - 1}^\star)^2 } - \frac{1}{n} \rb \conp 0, \label{z20}
		\end{align}
		where the projection matrices are defined in \eqref{def_projection_matrix}. 
	\end{lemma} 
		\begin{proof}[\textbf{\upshape Proof of Lemma  \ref{lem_conv_pii_sq}.}]
		As a preparation, we will first show that for any sequence $(i_n)_{n\in\N}$ such that $2 \leq i_n \leq p_n$ for all $n\in\N$ and the limit $\lim_{n\to\infty} i_n / n \in [0,1)$ exists, it holds
		\begin{align} \label{aim}
			 a_{i_n,n} - b_{i_n,n} \conp 0, ~ n\to\infty,
		\end{align}
		where we define for $2 \leq i \leq p$
		\begin{align*}
			a_{i,n} & = \frac{ \tr \lb \bfP(i -1)^{ \odot 2}   \rb  }{  n - i + 1 } , ~ 
			b_{i,n} = \lb 1 - \frac{i}{n} \rb, ~~
			c_n = \sum_{i=p_1 + 1}^p \frac{ a_{i, n} - b_{i,n} }{n - i +1}.
		\end{align*}
		In the following, we denote the diagonal entries of $\bfP(i_n -1)$ by $p_{ii}$ ($1 \leq i \leq n$). 
		First, we consider the case $\lim_{n\to\infty} i_n / n = 0.$ For this case, we note that
	\begin{align*} 
		\frac{\tr \lb \bfP(i_n - 1)^{\odot 2} \rb}{n} 
		& = \frac{1}{n} \sum\limits_{i=1}^n \lb 1 - p_{ii}\rb^2
		- 1 + \frac{2}{n} \sum\limits_{i=1}^n p_{ii}
		=  \frac{2 ( n - i_n + 1 ) }{n} - 1 + o_{\textnormal{Pr}}(1) \nonumber \\
		& = 1 + o_{\textnormal{Pr}}(1) 
		,~n\to\infty, 
	\end{align*}
	where we used
	\begin{align*}
		\frac{1}{n} \sum\limits_{i=1}^n \E (1 - p_{ii})^2 \leq 
		\frac{1}{n} \sum\limits_{i=1}^n \E [1 - p_{ii}] = \frac{1}{n} \tr ( \mathbf{I} - \mathbf{P} ) = \frac{i_n- 1}{n} = o(1), ~n\to\infty. 
	\end{align*}
	In this case, we conclude
	\begin{align*}
		a_{i_n,n} - b_{i_n,n} = \frac{\tr \lb \bfP(i_n - 1)^{\odot 2} \rb }{n} - 1  + o_{\PR}(1) = o_{\PR}(1), ~n\to\infty.
	\end{align*}
	Now consider the case $\lim_{n\to\infty} i_n / n = \gamma \in (0,1)$. 
	Then we have from Theorem 3.2 in \cite{anatolyev_yaskov_2017}
	\begin{align*}
		\frac{1}{n} \sum\limits_{i=1}^n (1 - p_{ii} - \gamma)^2 \conp 0, ~ n\to\infty,
\end{align*}	 
	which implies
	\begin{align*}
		\frac{\tr \lb \bfP(i_n - 1)^{\odot 2} \rb}{n} 
		& =  \frac{1}{n} \sum_{i=1}^n ( 1 - p_{ii} - \gamma)^2 
		- (1-\gamma)^2 + \frac{2 ( 1 -\gamma )}{n} \sum\limits_{i=1}^n p_{ii} \\
		& = \frac{2 ( 1 -\gamma ) ( n - i_n + 1) }{n} - (1 - \gamma)^2 + o_{\textnormal{Pr}}(1) 
		 = ( 1 - \gamma)^2 + o_{\textnormal{Pr}}(1), ~ n\to\infty,
	\end{align*}
		which implies that \eqref{aim} holds also true in this case. We continue with a proof of \eqref{z10} by showing that any subsequence of $(c_n)_{n\in\N}$ admits a further subsequence converging in probability to $0$.	
		Let $(c_{n_j})_{j\in\N}$ be an arbitrary subsequence of $(c_n)_{n\in\N}$. We choose
			\begin{align*} 
			 i_{n_j} \in \argmax_{p_1 + 1 \leq i \leq p} \lb a_{i,n_j}  - b_{i,n_j} \rb.
			 \end{align*} 		
			 Not that there exists a subsequence $(i_{n_{j_k}} )_{k\in\N}$ of $(i_{n_j})_{j\in\N}$ 
			 which admits a limit $\lim_{k\to\infty} i_{n_{j_k}}/k \in  [0,1)$ (that is, this subsequence satisfies the assumption for \eqref{aim}).
		Then, it holds using \eqref{aim}
		\begin{align*}
			c_{n_{j_k}} 
			\lesssim \max\limits_{p_1 +1 \leq i \leq p} \lb a_{i,n_{j_k}}  - b_{i,n_{j_k}} \rb
			= a_{i_{n_{j_k}},n_{j_k}}  - b_{i_{n_{j_k}},n_{j_k}} \conp 0,~ n\to\infty.
		\end{align*} 
		This implies the convergence $ c_n \conp 0$ of the whole sequence $(c_n)_{n\in\N}$ for $n\to\infty$ and thus, the convergence in \eqref{z1} holds true. The second assertion \eqref{z20} of Lemma \ref{lem_conv_pii_sq} can be shown similarly.  
		\end{proof}

	\begin{lemma} \label{lem_conv_sigma2i}
	It holds 
	\begin{align*}
			\sum\limits_{i=p_1 + 1}^p \sigma_{n,2,i}^2 \conp 0, ~ n\to\infty,
		\end{align*}
	where the term $\sigma_{n,2,i}^2$ is defined in \eqref{def_sigma_2}.
	\end{lemma} 
		\begin{proof}[\textbf{\upshape Proof of Lemma \ref{lem_conv_sigma2i}.}]
	 	Recalling the definition of $\sigma_{n,2,i}^2$ and using Lemma \ref{lem_conv_pii_sq}, it suffices to show
	 	\begin{align*}
	 		\sum\limits_{i= p_1 + 1}^p \lb \tr \mathbf{A}_i \odot \mathbf{B}_i - \frac{1}{n} \rb 
	 		\conp 0, ~ n\to\infty,
	 	\end{align*}
	 	where we denote
	 	\begin{align*}
	 		\mathbf{A}_i = \frac{\bfP(i-1)}{n - i + 1}, ~~
			\mathbf{B}_i = \frac{   \bfP( p_{g(i)-1}^\star+1 ;i-1)   } {  n - i + 1 + p_{g(i) - 1}^\star }, ~
			p_1 + 1 \leq i \leq p. 
	 	\end{align*}
	 	Note that $\tr \mathbf{A}_i = \tr \mathbf{B}_i = 1$ for all $p_1 + 1 \leq i \leq p$. This gives
	 	\begin{align*}
	 		 \sum\limits_{i= p_1 + 1}^p \lb \tr \mathbf{A}_i \odot \mathbf{B}_i - \frac{1}{n} \rb 
	 		 & = 
	 		  \sum\limits_{i= p_1 + 1}^p  \tr \left\{ \lb \mathbf{A}_i - \frac{1}{n} \mathbf{I} \rb \odot \lb \mathbf{B}_i - \frac{1}{n} \mathbf{I} \rb \right\} 
	 		   \leq \sum\limits_{i= p_1 + 1}^p \left\{  \tr  \lb \mathbf{A}_i - \frac{1}{n} \bfI \rb^{\odot 2} \tr  \lb \mathbf{B}_i - \frac{1}{n} \bfI \rb^{\odot 2} \right\}^{\frac{1}{2}}  \\
	 		  & \leq \left\{ \sum\limits_{i= p_1 + 1}^p   \tr  \lb \mathbf{A}_i - \frac{1}{n} \bfI \rb^{\odot 2} \sum\limits_{i= p_1 + 1}^p \tr  \lb \mathbf{B}_i - \frac{1}{n} \bfI \rb^{\odot 2} \right\}^{\frac{1}{2}}  
	 		   = \left\{ \sum\limits_{i= p_1 + 1}^p   \lb \tr   \mathbf{A}_i^{\odot 2}  - \frac{1}{n} \rb \sum\limits_{i= p_1 + 1}^p \lb  \tr   \mathbf{B}_i^{\odot 2}  - \frac{1}{n} \rb \right\}^{\frac{1}{2}}  \\ 
	 		  & \conp 0, 
	 	\end{align*}
	 	as $n\to\infty$, where we applied the Cauchy-Schwarz inequality twice and Lemma \ref{lem_conv_pii_sq}.  
	 	\end{proof}
	 	
	 		\begin{proof}[\textbf{\upshape Proof of Theorem \ref{thm_eq_cov}.}	]
		For proving Theorem \ref{thm_eq_cov}, we need the following properties of the variance.
\begin{lemma} \label{lem_sigma_bound_eq_cov}
	 Under the assumptions of Theorem \ref{thm_eq_cov},
	 we have 
	 \begin{align*}
	 	0 < \inf_{n\in\N} \sigma_n^2 \leq \sup_{n\in\N} \sigma_n^2 < \infty,
	 \end{align*}
	 where $\sigma_n^2$ denotes the variance defined in \eqref{def_sigma_eq_cov}. 
\end{lemma}	
	\begin{proof}[\textbf{\upshape Proof of Lemma \ref{lem_sigma_bound_eq_cov}. } ]
	Define the functions 
	\begin{align*}
		\xi(x) = - \lb \log (1 -x ) + x \rb, ~
		\eta(x) = \frac{\xi(x) }{x^2},
	\end{align*}
	where $x\in(0,1)$. Note that $\eta$ is a monotone increasing function with $\eta(x) \geq 1/2$.
	Using $\sum_{j=1}^q n_j = n$ and the definition $n_{\max} = \max_{1 \leq j \leq q} n_j$, we obtain the estimate 
	\begin{align*}
		n^2 \sigma_n^2 
		& =
		 \sum\limits_{j=1}^q n_j^2 \xi \lb \frac{p}{n_j} \rb 
		- n^2 \xi \lb \frac{p}{n} \rb 
	 = p^2 \left\{ \sum\limits_{j=1}^q \eta \lb \frac{p}{n_j } \rb - \eta \lb \frac{p}{n} \rb \right\} 	\\
	 & \geq  p^2 \left\{ q \eta \lb \frac{p}{n_{max}} \rb - \eta \lb \frac{p}{n_{\max}} \rb  \right\}  = p^2 ( q - 1) \eta \lb \frac{p}{n_{max}} \rb.
	\end{align*}
	Using $\inf_{n\in\N} p/n >0$, we conclude $\inf_{n\in\N} \sigma_{n}^2 > 0$. 
	Moreover, from our assumption on the dimension-to-subsample-size ratios $p/n_j$,  we obtain
	\begin{align*}
		\sup_{n\in\N} \sigma_n^2 \leq \sup_{n\in\N} \max\limits_{1 \leq j \leq q} \lb - \log \lb 1 - \frac{p}{n_j} \rb \rb 
		< \infty. 
	\end{align*}
	This concludes the proof of Lemma \ref{lem_sigma_bound_eq_cov}. 
	\end{proof}
	Continuing with the proof of Theorem \ref{thm_eq_cov}, we follow the same strategy as in the proof of Theorem \ref{thm} and concentrate on discussing the main steps. 
	Recalling the definition \eqref{def_statistic_2} of the likelihood ratio test statistic, we note that
	under the null hypothesis \eqref{null_eq_cov}
	\begin{align*}
		2 \log \Lambda_{n,1} 
		= &  \sum\limits_{j=1}^q n_j \log | \mathbf{A}_j | - n \log |\mathbf{A}| 
		+ p n \log n - \sum\limits_{j=1}^q n_j p \log n_j \\
		= & \sum\limits_{j=1}^q n_j \log | n_j \hat{\mathbf{I}}_j | - n \log |n \hat{\mathbf{I}}| 
		+ p n \log n - \sum\limits_{j=1}^q n_j p \log n_j ,
	\end{align*}
	where $\hat{\mathbf{I}}$ is defined in the proof of Theorem \ref{thm} and 
	\begin{align*}
		\hat{\mathbf{I}}_j = & \frac{1}{n_j} \sum\limits_{k=1}^{n_j} \bfx_{jk} \bfx_{jk}^\top. 
	\end{align*}
	Applying the QR-procedure to the matrices $\hat{\mathbf{I}}$ and $\hat{\mathbf{I}}_j$ ($1 \leq j \leq q$), we obtain for their determinants
	\begin{align*}
		| n \hat{\mathbf{I}} | = \prod\limits_{i=1}^p \bfb_i ^\top \bfP(i-1) \bfb_i, ~
		| n_j \hat{\mathbf{I}}_j | = \prod\limits_{i=1}^p \bfb_{ji}^\top \bfP (j; i-1) \bfb_{ji} ,
	\end{align*}
	where
\ND{ 	\begin{align*}
		\bfb_i = & \lb \bfb_{1i}^\top, \ldots, \bfb_{qi}^\top \rb ^\top \in \R^n, 
		\bfb_{ji} \in\R^{n_j}, ~ 1 \leq j \leq q, 
	\end{align*} }
	and throughout this proof, $\bfP(j;i-1)\in \R^{n_j \times n_j}$ denotes the projection matrix on the orthogonal complement of $$\textnormal{span}\{ \bfb_{j1} , \ldots, \bfb_{j,i-1}\}$$ for $1 \leq i \leq p,~ 1\leq j \leq q$ (note that we have a different definition than in the proof of Theorem \ref{thm}). We set $\bfP(j; 0) = \mathbf{I}\in \R^{n_j \times n_j}$. 
The remaining quantities are defined as in the proof of Theorem \ref{thm}. 
With a slight abuse of notation, we define for $1 \leq i \leq p,~ 1\leq j \leq q$
\begin{align*}
	X_i  =  & \frac{ \bfb_i^\top \bfP(i-1) \bfb_i - (n - i + 1 ) }{ n - i +1}, ~
		X_{j,i}  =  \frac{ \bfb_{ji}^\top \bfP(j;i-1) \bfb_{ji} - (n_j - i + 1 ) }{ n_j - i + 1 },
		\\ 
		Y_i = & \log ( 1 + X_i) - \lb X_i - \frac{X_i^2}{2} \rb, ~ 
		Y_{j,i} =  \log ( 1 + X_{j,i}) - \lb X_{j,i} - \frac{X_{j,i}^2}{2} \rb.
		\end{align*}
		Similarly to \eqref{cal_mu}, we obtain using Stirling's formula
		\begin{align}
			& \sum\limits_{i=1}^p \sum\limits_{j=1}^q n_j \log ( n_j - i + 1) - n \sum\limits_{i=1}^p \log ( n - i + 1 ) 
			+ pn\log n - \sum\limits_{j=1}^q n_j p \log n_j \nonumber \\
			= & \sum\limits_{j=1}^q n_j \log \lb \frac{n_j ! }{(n_j - p )! } \rb 
			- n \log \lb \frac{n!}{(n - p)! } \rb 
			+ pn\log n - \sum\limits_{j=1}^q n_j p \log n_j 
			\nonumber \\ 
			= & n \lb n - p + \frac{1}{2} \rb \log \lb 1 - \frac{p}{n} \rb - \sum\limits_{j=1}^q 
			n_j \lb n_j - p + \frac{1}{2} \rb \log \lb 1 - \frac{p}{n_j} \rb 
			 + \sum_{j=1}^q \frac{n_j}{12} \lb \frac{1}{n_j} - \frac{1}{n_j - p} \rb 
			+ o(n) \nonumber \\
				= & n \lb n - p + \frac{1}{2} \rb \log \lb 1 - \frac{p}{n} \rb - \sum\limits_{j=1}^q 
			n_j \lb n_j - p + \frac{1}{2} \rb \log \lb 1 - \frac{p}{n_j} \rb 
			 - \sum_{j=1}^q \frac{p}{12 ( n_j - p) } 
			+ o(n) \nonumber \\
			= & n \lb n - p + \frac{1}{2} \rb \log \lb 1 - \frac{p}{n} \rb - \sum\limits_{j=1}^q 
			n_j \lb n_j - p + \frac{1}{2} \rb \log \lb 1 - \frac{p}{n_j} \rb 
			+ o(n) \nonumber \\ 
			& =  \mu_n 
            + n  \log \lb 1 - \frac{p}{n} \rb 
            - \sum\limits_{j=1}^q 
			n_j  \log \lb 1 - \frac{p}{n_j} \rb  
            - \frac{1}{2}(\nu_4 -3) p(1-q)
 + o(n), ~ n\to\infty.
 \label{z4}
		\end{align}
	
		For \eqref{z4}, we used that under the assumptions of Theorem \ref{thm_eq_cov}
		\begin{align}
			\sum_{j=1}^q \frac{p}{12 n ( n_j - p) } \lesssim \frac{q}{n} \leq \frac{1}{\min_j n_j }  =o(1), ~n\to\infty. \label{z5}
		\end{align}
	Consequently, we may decompose
	\begin{align*}
		& 2 \lb \log \Lambda_{n,1} - \mu_n \rb 
		\\ = & \sum\limits_{j=1}^q \sum\limits_{i=1}^{p} n_j X_{j,i} 
		-  n \sum\limits_{i=1}^p X_i 
		-  \lb    \sum\limits_{j=1}^q  \sum\limits_{i=1}^{p}n_j \frac{X_{j,i}^2}{2} 
		- n \sum\limits_{i=1}^p \frac{X_i^2}{2} 
		- \tau_n  
		\rb 
		 +  \sum\limits_{j=1}^q   \sum\limits_{i=1}^{p} n_jY_{j,i} 
		- n \sum\limits_{i=1}^p Y_i 
		+ o(n), ~n\to\infty, 
	\end{align*}
     where
		\begin{align}
			\tau_n = & 2 \left\{   n \log \lb 1 - \frac{p}{n} \rb - \sum\limits_{j=1}^q n_j \log \lb 1 - \frac{p}{n_j} \rb \right\} + (\nu_4 -3) p(q-1).
			\nonumber 
		\end{align}
Note that $(W_i)_{1 \leq i \leq p}$ with
\begin{align*}
	W_i = \sum\limits_{j=1}^q n_j X_{j,i} - n X_i  , ~ 1 \leq i \leq p,
\end{align*}
forms a martingale difference scheme with respect to the filtration $(\mathcal{A}_i)_{1 \leq i \leq p}$, where the $\sigma$-field $\mathcal{A}_i$ is generated by the random variables $\bfb_1, \ldots, \bfb_{i}$ for $1 \leq i \leq p$. One can show that 
	\begin{align} \label{z1}
		\frac{\sum\limits_{j=1}^q  \sum\limits_{i=1}^{p}n_j \frac{X_{j,i}^2}{2} 
		- n \sum\limits_{i=1}^p \frac{X_i^2}{2} 
		- \tau_n }{n \sigma_n}
		\conp 0, ~ n \to\infty, 
	\end{align}
	 and
	\begin{align} \label{z2}
		\frac{ \sum\limits_{j=1}^q   \sum\limits_{i=1}^{p} n_jY_{j,i} 
		- n \sum\limits_{i=1}^p Y_i }{n\sigma_n} \conp 0,~ n\to\infty.
	\end{align}
 For the sake of brevity, we omit the proofs of \eqref{z1} and \eqref{z2} as they are very similar to the proofs of Lemma \ref{lem_conv_x_pow2} and Lemma \ref{lem_conv_log_term}. For \eqref{z1}, we additionally note that (similarly to Lemma \ref{lem_conv_sigma2i})
 \begin{align*}
     n \sum_{i=1}^p \frac{\tr \bfP (i-1) ^{\odot 2}}{(n - i +1)^2} - \sum_{j=1}^q n_j \sum_{i=1}^p \frac{\tr \bfP (j;i-1) ^{\odot 2}}{(n_j - i +1)^2}
     = \frac{np}{n} - \sum_{j=1}^q \frac{n_j p}{n_j} + o_{\PR}(n)
     = p - qp + o_{\PR}(n) = p(1-q) + o_{\PR}(n) .
 \end{align*}
 
 We continue with a proof of the asymptotic normality of the scheme $(W_i / (n\sigma_n))_{1 \leq i \leq p}. $ To begin with, we show that
\begin{align}
	\sum\limits_{i=1}^p \E  \left[ \lb \frac{W_i}{n \sigma_n}\rb^2 \Bigg| \mathcal{A}_{i-1} \right]
	= 1 + o_{\PR} (1), ~ n\to\infty. \label{conv_var_2}
\end{align}
As a preparation for \eqref{conv_var_2}, note that 
\begin{align*}
	&  \sum\limits_{i=1}^p \sum\limits_{j=1}^q n_j n \E [ X_i X_{j,i} | \mathcal{A}_{i-1} ] 
	 =  \sum\limits_{i=1}^p \sum\limits_{j=1}^q n n_j \left\{ \sum\limits_{\substack{k = n_{j-1}^\star + 1}}^{n_j^\star} ( \nu_4 - 3)\frac{ \lb \bfP(j; i-1) \rb_{kk} \lb \bfP(i-1) \rb_{kk} }{(n - i + 1) (n_j - i + 1) } 
	+ 2 \frac{ \tr \lb \bfP(j-1) \tilde{ \bfP } ( j; i- 1 ) \rb }{ ( n - i + 1) (n_j - i + 1)} \right\}  \\
	& = 	 \sum\limits_{i=1}^p \sum\limits_{j=1}^q n n_j \left\{  ( \nu_4 - 3)\frac{ \tr \lb \tilde \bfP(j; i-1)\odot \bfP(i-1) \rb }{(n - i + 1) (n_j - i + 1) } 
	+ 2 \frac{ \tr  \tilde{ \bfP } ( j; i- 1 ) }{ (n_j - i  +1) (n - i + 1)} \right\}  \\
	& = ( \nu_4 - 3) \sum\limits_{i=1}^p \sum\limits_{j=1}^q n n_j  \frac{ \tr \lb \tilde \bfP(j; i-1) \odot \bfP(i-1) \rb }{(n - i + 1) (n_j - i + 1) } 
	+ \sum\limits_{i=1}^p \frac{2 n^2 }{(n - i + 1)}, \\
\end{align*}
where $\tilde{\bfP}(j;i-1) $ denotes the $(n\times n)$ dimensional embedded matrix of $\bfP(j; i -1) \in \R^{n_j \times n_j}$, that is, 
\begin{align*}
	\lb \tilde \bfP(j;i-1) \rb_{kl} 
	=
	\begin{cases} 
	 \lb  \bfP(j;i-1) \rb_{kl},  &  n_{j-1}^\star + 1 \leq k,l \leq n_j^\star , \\
	 0 ,& \textnormal{else} ,
	 \end{cases}
\end{align*}
for $1 \leq k , l \leq n, ~ 1 \leq j \leq q, ~ 1 \leq i \leq p.$
In order to prove \eqref{conv_var_2}, we calculate 
	\begin{align*}
		\sum\limits_{i=1}^p \E  \left[ W_i^2 \big| \mathcal{A}_{i-1} \right]
		= & n^2 \sum\limits_{i=1}^p \E [ X_i^2 | \mathcal{A}_{i-1} ] 	
		+ \sum\limits_{i=1}^p \sum\limits_{j=1}^q n_j^2 \E [ X_{j,i}^2 | \mathcal{A}_{i-1} ]
		- 2 \sum\limits_{i=1}^p \sum\limits_{j=1}^q n_j n \E [ X_i X_{j,i} | \mathcal{A}_{i-1} ] 
		\\
		= & 	 2 \sum\limits_{j=1}^q \sum\limits_{i=1}^p n_j^2 \frac{1}{n_j - i + 1} 
		- 2 n^2 \sum\limits_{i=1}^p \frac{1}{n - i + 1}
		  + \tilde{\sigma}_{n,2}^2 ,
	\end{align*}
		where we used the fact $\E [ X_{j,i} X_{j',i} | \mathcal{A}_{i-1} ] = 0$ for different groups $j,j' \in \{1, \ldots, q\}, ~ j\neq j' $ and \eqref{var_quad_form} and we define
	\begin{align*}
		  \tilde{\sigma}_{n,2}^2 = & ( \nu_4 - 3 ) \Bigg\{ n^2 \sum\limits_{i=1}^p \frac{ \tr \bfP(i-1)^{\odot 2} }{ (n - i +1)^2 } 
		+ \sum\limits_{i=1}^p \sum\limits_{j=1}^q n_j^2 \frac{ \tr \bfP(j; i-1)^{\odot 2} }{(n_j - i + 1)^2}  
		 - 2 \sum\limits_{i=1}^p \sum\limits_{j=1}^q n n_j \frac{ \tr \lb \tilde  \bfP(j; i-1) \odot \bfP(i-1) \rb }{(n - i + 1) (n_j - i + 1) } \Bigg\}.
	\end{align*}
	Note that $\tilde \sigma_{n,2}^2 / n^2 = o_{\PR} (1)$ as $n\to\infty$ (similarly to Lemma \ref{lem_conv_pii_sq} and  Lemma \ref{lem_conv_sigma2i}).  
	Using \eqref{harmonic_series} and \eqref{z5}, we obtain
	\begin{align*}
		\sum\limits_{i=1}^p \E  \left[ \lb \frac{W_i}{n}\rb^2 \big| \mathcal{A}_{i-1} \right]
		= &  \sum\limits_{i=1}^p \log \lb 1 - \frac{p}{n} \rb 
		- \sum\limits_{i=1}^p \sum\limits_{j=1}^q \lb \frac{n_j}{n} \rb^2 \log \lb 1 - \frac{p}{n_j} \rb 
		 + \sum\limits_{j=1}^q \lb \frac{n_j}{n} \rb^2 \left\{ \frac{1}{2 n_j} - \frac{1}{2 (n_j - p) } \right\} 
		+ \frac{ \tilde{\sigma}_{n,2}^2 }{n^2} +  o(1) \\
		= & \sigma_n^2 + o_{\PR}(1), ~n\to\infty,
	\end{align*}
	which implies \eqref{conv_var_2} by an application of Lemma \ref{lem_sigma_bound_eq_cov}. 
	Note that the Lindeberg condition for the scheme $(W_i / (n\sigma_n) )_{1\leq i \leq p}$ can be shown similarly to \eqref{lindeberg} using Lemma \ref{lem_sigma_bound_eq_cov}.  
	Combining \eqref{z1} and \eqref{z2}, we conclude  
	\begin{align*}
		& \frac{ 2 \lb \log \Lambda_{n,1} - \mu_n \rb }{n\sigma_n}
		 = \frac{ \sum\limits_{j=1}^q \sum\limits_{i=1}^{p} n_j X_{j,i} 
		-  n \sum\limits_{i=1}^p X_i }{ n \sigma_n} 
		+ o_{\PR}(1), ~ n\to\infty.
	\end{align*}
		By an application of Corollary 3.1 of \cite{hall_heyde}, we obtain
	\begin{align*} 
			 \frac{ \sum\limits_{j=1}^q \sum\limits_{i=1}^{p} n_j X_{j,i} 
		-  n \sum\limits_{i=1}^p X_i }{ n \sigma_n} 
		\cond \mathcal{N} (0,1), ~ n\to\infty. 
	\end{align*}
	The proof Theorem \ref{thm_eq_cov} concludes. 
	\end{proof}
\section*{Acknowledgments}
This work was supported by the German Research Foundation (DFG Research Unit 1735, DE 502/26-2, RTG 2131, \textit{High-dimensional Phenomena in Probability - Fluctuations and Discontinuity} and project number 460867398, DFG Research unit 5381, {\it Mathematical Statistics in the Information Age}). 
The author would like to thank Tim Kutta for his helpful comments on an earlier version of this manuscript. The very constructive comments of the Editor, Associate Editor and referees are kindly acknowledged.


\bibliographystyle{myjmva}
\bibliography{references}

\begin{thebibliography}{52}
\expandafter\ifx\csname natexlab\endcsname\relax\def\natexlab#1{#1}\fi
\providecommand{\bibinfo}[2]{#2}
\ifx\xfnm\relax \def\xfnm[#1]{\unskip,\space#1}\fi
\bibitem[{Anatolyev and Yaskov(2017)}]{anatolyev_yaskov_2017}
\bibinfo{author}{S.~Anatolyev}, \bibinfo{author}{P.~Yaskov},
  \bibinfo{title}{Asymptotics of diagonal elements of projection matrices under
  many instruments/regressors}, \bibinfo{journal}{Econometric Theory}
  \bibinfo{volume}{33} (\bibinfo{year}{2017}) \bibinfo{pages}{717--738}.
\bibitem[{Anderson(1984)}]{anderson2003}
\bibinfo{author}{T.~W. Anderson}, \bibinfo{title}{An introduction to
  multivariate statistical analysis}, Wiley Series in Probability and
  Mathematical Statistics: Probability and Mathematical Statistics,
  \bibinfo{publisher}{John Wiley \& Sons, Inc., New York},
  \bibinfo{edition}{second} edition, \bibinfo{year}{1984}.
\bibitem[{Bai and Silverstein(2010)}]{bai2004}
\bibinfo{author}{Z.~Bai}, \bibinfo{author}{J.~W. Silverstein},
  \bibinfo{title}{Spectral analysis of large dimensional random matrices},
  volume~\bibinfo{volume}{20}, \bibinfo{publisher}{Springer},
  \bibinfo{year}{2010}.
\bibitem[{Bao et~al.(2017)Bao, Hu, Pan and Zhou}]{bao2017test}
\bibinfo{author}{Z.~Bao}, \bibinfo{author}{J.~Hu}, \bibinfo{author}{G.~Pan},
  \bibinfo{author}{W.~Zhou}, \bibinfo{title}{Test of independence for
  high-dimensional random vectors based on freeness in block correlation
  matrices}, \bibinfo{journal}{Electronic Journal of Statistics}
  \bibinfo{volume}{11} (\bibinfo{year}{2017}) \bibinfo{pages}{1527--1548}.
\bibitem[{Bao et~al.(2015)Bao, Pan and Zhou}]{bao2015}
\bibinfo{author}{Z.~Bao}, \bibinfo{author}{G.~Pan}, \bibinfo{author}{W.~Zhou},
  \bibinfo{title}{The logarithmic law of random determinant},
  \bibinfo{journal}{Bernoulli} \bibinfo{volume}{21} (\bibinfo{year}{2015})
  \bibinfo{pages}{1600--1628}.
\bibitem[{Barrick and Mount(1991)}]{barrick1991big}
\bibinfo{author}{M.~R. Barrick}, \bibinfo{author}{M.~K. Mount},
  \bibinfo{title}{The big five personality dimensions and job performance: a
  meta-analysis}, \bibinfo{journal}{Personnel psychology} \bibinfo{volume}{44}
  (\bibinfo{year}{1991}) \bibinfo{pages}{1--26}.
\bibitem[{Bodnar et~al.(2019)Bodnar, Dette and Parolya}]{bodnar2019}
\bibinfo{author}{T.~Bodnar}, \bibinfo{author}{H.~Dette},
  \bibinfo{author}{N.~Parolya}, \bibinfo{title}{Testing for independence of
  large dimensional vectors}, \bibinfo{journal}{The Annals of Statistics}
  \bibinfo{volume}{47} (\bibinfo{year}{2019}) \bibinfo{pages}{2977--3008}.
\bibitem[{Carson(1998)}]{carson1998integration}
\bibinfo{author}{A.~D. Carson}, \bibinfo{title}{The integration of interests,
  aptitudes, and personality traits: A test of lowman's matrix},
  \bibinfo{journal}{Journal of Career Assessment} \bibinfo{volume}{6}
  (\bibinfo{year}{1998}) \bibinfo{pages}{83--105}.
\bibitem[{Cattell and Mead(2008)}]{cattell2008sixteen}
\bibinfo{author}{H.~E. Cattell}, \bibinfo{author}{A.~D. Mead},
  \bibinfo{title}{The sixteen personality factor questionnaire (16pf)},
  \bibinfo{journal}{The SAGE handbook of personality theory and assessment}
  \bibinfo{volume}{2} (\bibinfo{year}{2008}) \bibinfo{pages}{135--159}.
\bibitem[{Cattell(1957)}]{cattell1957personality}
\bibinfo{author}{R.~B. Cattell}, \bibinfo{title}{Personality and motivation
  structure and measurement.}  (\bibinfo{year}{1957}).
\bibitem[{Cattell(1973)}]{cattell1973}
\bibinfo{author}{R.~B. Cattell}, \bibinfo{title}{Personality and mood by
  questionnaire}, \bibinfo{publisher}{Jossey-Bass}, \bibinfo{year}{1973}.
\bibitem[{Cattell and Krug(1986)}]{cattell1986number}
\bibinfo{author}{R.~B. Cattell}, \bibinfo{author}{S.~E. Krug},
  \bibinfo{title}{The number of factors in the 16pf: A review of the evidence
  with special emphasis on methodological problems},
  \bibinfo{journal}{Educational and Psychological Measurement}
  \bibinfo{volume}{46} (\bibinfo{year}{1986}) \bibinfo{pages}{509--522}.
\bibitem[{De~Raad(2000)}]{de2000big}
\bibinfo{author}{B.~De~Raad}, \bibinfo{title}{The big five personality factors:
  the psycholexical approach to personality.}, \bibinfo{publisher}{Hogrefe \&
  Huber Publishers}, \bibinfo{year}{2000}.
\bibitem[{Dette and D{\"o}rnemann(2020)}]{dette2020}
\bibinfo{author}{H.~Dette}, \bibinfo{author}{N.~D{\"o}rnemann},
  \bibinfo{title}{Likelihood ratio tests for many groups in high dimensions},
  \bibinfo{journal}{Journal of Multivariate Analysis} \bibinfo{volume}{178}
  (\bibinfo{year}{2020}) \bibinfo{pages}{104605}.
\bibitem[{Fan and Li(2006)}]{Fan2006}
\bibinfo{author}{J.~Fan}, \bibinfo{author}{R.~Li}, \bibinfo{title}{Statistical
  challenges with high dimensionality: Feature selection in knowledge
  discovery}, \bibinfo{journal}{Proceedings of the International Congress of
  Mathematicians, Madrid} \bibinfo{volume}{3} (\bibinfo{year}{2006}).
\bibitem[{Gao et~al.(2017)Gao, Han, Pan and Yang}]{gao2017}
\bibinfo{author}{J.~Gao}, \bibinfo{author}{X.~Han}, \bibinfo{author}{G.~Pan},
  \bibinfo{author}{Y.~Yang}, \bibinfo{title}{High dimensional correlation
  matrices: The central limit theorem and its applications},
  \bibinfo{journal}{Journal of the Royal Statistical Society: Series B
  (Statistical Methodology)} \bibinfo{volume}{79} (\bibinfo{year}{2017})
  \bibinfo{pages}{677--693}.
\bibitem[{Girko(1998)}]{girko1998refinement}
\bibinfo{author}{V.~L. Girko}, \bibinfo{title}{A refinement of the central
  limit theorem for random determinants}, \bibinfo{journal}{Theory of
  Probability \& Its Applications} \bibinfo{volume}{42} (\bibinfo{year}{1998})
  \bibinfo{pages}{121--129}.
\bibitem[{Gosling et~al.(2003)Gosling, Rentfrow and Swann~Jr}]{gosling2003very}
\bibinfo{author}{S.~D. Gosling}, \bibinfo{author}{P.~J. Rentfrow},
  \bibinfo{author}{W.~B. Swann~Jr}, \bibinfo{title}{A very brief measure of the
  big-five personality domains}, \bibinfo{journal}{Journal of Research in
  personality} \bibinfo{volume}{37} (\bibinfo{year}{2003})
  \bibinfo{pages}{504--528}.
\bibitem[{Guo and Qi(2021)}]{guo2021}
\bibinfo{author}{W.~Guo}, \bibinfo{author}{Y.~Qi}, \bibinfo{title}{Asymptotic
  distributions for likelihood ratio tests for the equality of covariance
  matrices}, \bibinfo{journal}{arXiv preprint arXiv:2110.02384}
  (\bibinfo{year}{2021}).
\bibitem[{Hall and Heyde(1980)}]{hall_heyde}
\bibinfo{author}{P.~Hall}, \bibinfo{author}{C.~C. Heyde},
  \bibinfo{title}{Martingale limit theory and its application},
  \bibinfo{publisher}{Academic press}, \bibinfo{year}{1980}.
\bibitem[{Han et~al.(2017)Han, Chen and Liu}]{han2017distribution}
\bibinfo{author}{F.~Han}, \bibinfo{author}{S.~Chen}, \bibinfo{author}{H.~Liu},
  \bibinfo{title}{Distribution-free tests of independence in high dimensions},
  \bibinfo{journal}{Biometrika} \bibinfo{volume}{104} (\bibinfo{year}{2017})
  \bibinfo{pages}{813--828}.
\bibitem[{He et~al.(2021)He, Xu, Wu and Pan}]{he2021asymptotically}
\bibinfo{author}{Y.~He}, \bibinfo{author}{G.~Xu}, \bibinfo{author}{C.~Wu},
  \bibinfo{author}{W.~Pan}, \bibinfo{title}{Asymptotically independent
  u-statistics in high-dimensional testing}, \bibinfo{journal}{The Annals of
  Statistics} \bibinfo{volume}{49} (\bibinfo{year}{2021})
  \bibinfo{pages}{154--181}.
\bibitem[{Heiny and Parolya(2021)}]{heiny2021}
\bibinfo{author}{J.~Heiny}, \bibinfo{author}{N.~Parolya}, \bibinfo{title}{Log
  determinant of large correlation matrices under infinite fourth moment},
  \bibinfo{journal}{arXiv preprint arXiv:2112.15388}  (\bibinfo{year}{2021}).
\bibitem[{Hu et~al.(2017)Hu, Bai, Wang and Wang}]{hu2017testing}
\bibinfo{author}{J.~Hu}, \bibinfo{author}{Z.~Bai}, \bibinfo{author}{C.~Wang},
  \bibinfo{author}{W.~Wang}, \bibinfo{title}{On testing the equality of high
  dimensional mean vectors with unequal covariance matrices},
  \bibinfo{journal}{Annals of the Institute of Statistical Mathematics}
  \bibinfo{volume}{69} (\bibinfo{year}{2017}) \bibinfo{pages}{365--387}.
\bibitem[{Ishii et~al.(2021)Ishii, Maruo, Noma and
  Gosho}]{ishii2021statistical}
\bibinfo{author}{R.~Ishii}, \bibinfo{author}{K.~Maruo},
  \bibinfo{author}{H.~Noma}, \bibinfo{author}{M.~Gosho},
  \bibinfo{title}{Statistical inference based on accelerated failure time
  models under model misspecification and small samples},
  \bibinfo{journal}{Statistics in Biopharmaceutical Research}
  \bibinfo{volume}{13} (\bibinfo{year}{2021}) \bibinfo{pages}{384--394}.
\bibitem[{Jiang et~al.(2013)Jiang, Bai and Zheng}]{jiang2013testing}
\bibinfo{author}{D.~Jiang}, \bibinfo{author}{Z.~Bai},
  \bibinfo{author}{S.~Zheng}, \bibinfo{title}{Testing the independence of sets
  of large-dimensional variables}, \bibinfo{journal}{Science China Mathematics}
  \bibinfo{volume}{56} (\bibinfo{year}{2013}) \bibinfo{pages}{135--147}.
\bibitem[{Jiang and Wang(2017)}]{jiang2017moderate}
\bibinfo{author}{H.~Jiang}, \bibinfo{author}{S.~Wang}, \bibinfo{title}{Moderate
  deviation principles for classical likelihood ratio tests of high-dimensional
  normal distributions}, \bibinfo{journal}{Journal of Multivariate Analysis}
  \bibinfo{volume}{156} (\bibinfo{year}{2017}) \bibinfo{pages}{57--69}.
\bibitem[{Jiang and Qi(2015)}]{jiang2015}
\bibinfo{author}{T.~Jiang}, \bibinfo{author}{Y.~Qi}, \bibinfo{title}{Likelihood
  ratio tests for high-dimensional normal distributions},
  \bibinfo{journal}{Scandinavian Journal of Statistics} \bibinfo{volume}{42}
  (\bibinfo{year}{2015}) \bibinfo{pages}{988--1009}.
\bibitem[{Jiang and Yang(2013)}]{jiang_yang_2013}
\bibinfo{author}{T.~Jiang}, \bibinfo{author}{F.~Yang}, \bibinfo{title}{Central
  limit theorems for classical likelihood ratio tests for high-dimensional
  normal distributions}, \bibinfo{journal}{The Annals of Statistics}
  \bibinfo{volume}{41} (\bibinfo{year}{2013}) \bibinfo{pages}{2029--2074}.
\bibitem[{Johnstone(2006)}]{Johnstone2006}
\bibinfo{author}{I.~M. Johnstone}, \bibinfo{title}{High dimensional statistical
  inference and random matrices}, \bibinfo{journal}{Proceedings of the
  International Congress of Mathematicians, Madrid}  (\bibinfo{year}{2006}).
\bibitem[{Krug and Johns(1990)}]{krug199016}
\bibinfo{author}{S.~Krug}, \bibinfo{author}{E.~Johns}, \bibinfo{title}{The 16
  personality factor questionnaire}, \bibinfo{journal}{Testing and Counseling
  Practice. Mahwah, NJ: Lawrence Erlbaum Associates}  (\bibinfo{year}{1990})
  \bibinfo{pages}{63--90}.
\bibitem[{Lemonte(2016)}]{lemonte2016gradient}
\bibinfo{author}{A.~Lemonte}, \bibinfo{title}{The gradient test: another
  likelihood-based test}, \bibinfo{publisher}{Academic Press},
  \bibinfo{year}{2016}.
\bibitem[{Lemonte(2013)}]{lemonte2013gradient}
\bibinfo{author}{A.~J. Lemonte}, \bibinfo{title}{On the gradient statistic
  under model misspecification}, \bibinfo{journal}{Statistics \& Probability
  Letters} \bibinfo{volume}{83} (\bibinfo{year}{2013})
  \bibinfo{pages}{390--398}.
\bibitem[{Li et~al.(2017)Li, Chen and Yao}]{li2017testing}
\bibinfo{author}{W.~Li}, \bibinfo{author}{J.~Chen}, \bibinfo{author}{J.~Yao},
  \bibinfo{title}{Testing the independence of two random vectors where only one
  dimension is large}, \bibinfo{journal}{Statistics} \bibinfo{volume}{51}
  (\bibinfo{year}{2017}) \bibinfo{pages}{141--153}.
\bibitem[{Li et~al.(2022)Li, Wang, Yao and Zhou}]{li2022}
\bibinfo{author}{W.~Li}, \bibinfo{author}{Q.~Wang}, \bibinfo{author}{J.~Yao},
  \bibinfo{author}{W.~Zhou}, \bibinfo{title}{On eigenvalues of a
  high-dimensional spatial-sign covariance matrix},
  \bibinfo{journal}{Bernoulli} \bibinfo{volume}{28} (\bibinfo{year}{2022})
  \bibinfo{pages}{606--637}.
\bibitem[{Loubaton and Rosuel(2021)}]{loubaton2020large}
\bibinfo{author}{P.~Loubaton}, \bibinfo{author}{A.~Rosuel},
  \bibinfo{title}{Properties of linear spectral statistics of
  frequency-smoothed estimated spectral coherence matrix of high-dimensional
  gaussian time series}, \bibinfo{journal}{Electronic Journal of Statistics}
  \bibinfo{volume}{15} (\bibinfo{year}{2021}) \bibinfo{pages}{5380--5454}.
\bibitem[{Luo and Tsai(2012)}]{luo2012proportional}
\bibinfo{author}{X.~Luo}, \bibinfo{author}{W.~Y. Tsai}, \bibinfo{title}{A
  proportional likelihood ratio model}, \bibinfo{journal}{Biometrika}
  \bibinfo{volume}{99} (\bibinfo{year}{2012}) \bibinfo{pages}{211--222}.
\bibitem[{Mershon and Gorsuch(1988)}]{mershon1988number}
\bibinfo{author}{B.~Mershon}, \bibinfo{author}{R.~L. Gorsuch},
  \bibinfo{title}{Number of factors in the personality sphere: Does increase in
  factors increase predictability of real-life criteria?},
  \bibinfo{journal}{Journal of Personality and Social Psychology}
  \bibinfo{volume}{55} (\bibinfo{year}{1988}) \bibinfo{pages}{675}.
\bibitem[{Mestre and Vallet(2017)}]{mestre2017correlation}
\bibinfo{author}{X.~Mestre}, \bibinfo{author}{P.~Vallet},
  \bibinfo{title}{Correlation tests and linear spectral statistics of the
  sample correlation matrix}, \bibinfo{journal}{IEEE Transactions on
  Information Theory} \bibinfo{volume}{63} (\bibinfo{year}{2017})
  \bibinfo{pages}{4585--4618}.
\bibitem[{Muirhead(1982)}]{muirhead1982}
\bibinfo{author}{R.~J. Muirhead}, \bibinfo{title}{Aspects of Multivariate
  Statistical Theory}, \bibinfo{publisher}{John Wiley \& Sons},
  \bibinfo{address}{New York}, \bibinfo{year}{1982}.
\bibitem[{Nguyen and Vu(2014)}]{nguyen_vu}
\bibinfo{author}{H.~H. Nguyen}, \bibinfo{author}{V.~Vu}, \bibinfo{title}{Random
  matrices: Law of the determinant}, \bibinfo{journal}{The Annals of
  Probability} \bibinfo{volume}{42} (\bibinfo{year}{2014})
  \bibinfo{pages}{146--167}.
\bibitem[{O'Brien(1992)}]{obrien1992}
\bibinfo{author}{P.~O'Brien}, \bibinfo{title}{Robust procedures for testing
  equality of covariance matrices}, \bibinfo{journal}{Biometrics}
  \bibinfo{volume}{48} (\bibinfo{year}{1992}) \bibinfo{pages}{819--827}.
\bibitem[{Parolya et~al.(2021)Parolya, Heiny and Kurowicka}]{parolya2021}
\bibinfo{author}{N.~Parolya}, \bibinfo{author}{J.~Heiny},
  \bibinfo{author}{D.~Kurowicka}, \bibinfo{title}{Logarithmic law of large
  random correlation matrix}, \bibinfo{journal}{arXiv preprint
  arXiv:2103.13900}  (\bibinfo{year}{2021}).
\bibitem[{Qi et~al.(2019)Qi, Wang and Zhang}]{qi_et_al_2019}
\bibinfo{author}{Y.~Qi}, \bibinfo{author}{F.~Wang}, \bibinfo{author}{L.~Zhang},
  \bibinfo{title}{Limiting distributions of likelihood ratio test for
  independence of components for high-dimensional normal vectors},
  \bibinfo{journal}{Annals of the Institute of Statistical Mathematics}
  \bibinfo{volume}{71} (\bibinfo{year}{2019}) \bibinfo{pages}{911--946}.
\bibitem[{Roccas et~al.(2002)Roccas, Sagiv, Schwartz and Knafo}]{roccas2002big}
\bibinfo{author}{S.~Roccas}, \bibinfo{author}{L.~Sagiv}, \bibinfo{author}{S.~H.
  Schwartz}, \bibinfo{author}{A.~Knafo}, \bibinfo{title}{The big five
  personality factors and personal values}, \bibinfo{journal}{Personality and
  social psychology bulletin} \bibinfo{volume}{28} (\bibinfo{year}{2002})
  \bibinfo{pages}{789--801}.
\bibitem[{Schott(2007)}]{schott2007}
\bibinfo{author}{J.~R. Schott}, \bibinfo{title}{A test for the equality of
  covariance matrices when the dimension is large relative to the sample
  sizes}, \bibinfo{journal}{Computational Statistics and Data Analysis}
  \bibinfo{volume}{51} (\bibinfo{year}{2007}) \bibinfo{pages}{6535--6542}.
\bibitem[{Srivastava and Yanagihara(2010)}]{srivastava2010}
\bibinfo{author}{M.~S. Srivastava}, \bibinfo{author}{H.~Yanagihara},
  \bibinfo{title}{Testing the equality of several covariance matrices with
  fewer observations than the dimension}, \bibinfo{journal}{Journal of
  Multivariate Analysis} \bibinfo{volume}{101} (\bibinfo{year}{2010})
  \bibinfo{pages}{1319--1329}.
\bibitem[{Strug(2018)}]{strug2018evidential}
\bibinfo{author}{L.~J. Strug}, \bibinfo{title}{The evidential statistical
  paradigm in genetics}, \bibinfo{journal}{Genetic Epidemiology}
  \bibinfo{volume}{42} (\bibinfo{year}{2018}) \bibinfo{pages}{590--607}.
\bibitem[{Wang et~al.(2018)Wang, Han and Pan}]{wang2018}
\bibinfo{author}{X.~Wang}, \bibinfo{author}{X.~Han}, \bibinfo{author}{G.~Pan},
  \bibinfo{title}{The logarithmic law of sample covariance matrices near
  singularity}, \bibinfo{journal}{Bernoulli} \bibinfo{volume}{24}
  (\bibinfo{year}{2018}) \bibinfo{pages}{80--114}.
\bibitem[{Watterson(2002)}]{watterson2002}
\bibinfo{author}{D.~Watterson}, \bibinfo{title}{The 16pf leadership coaching
  report manual}, \bibinfo{year}{2002}.
\bibitem[{Yamada et~al.(2017)Yamada, Hyodo and Nishiyama}]{yamadaetal2017}
\bibinfo{author}{Y.~Yamada}, \bibinfo{author}{M.~Hyodo},
  \bibinfo{author}{T.~Nishiyama}, \bibinfo{title}{Testing block-diagonal
  covariance structure for high-dimensional data under non-normality},
  \bibinfo{journal}{Journal of Multivariate Analysis} \bibinfo{volume}{155}
  (\bibinfo{year}{2017}) \bibinfo{pages}{305--316}.
\bibitem[{Yang et~al.(2021)Yang, Zheng and Chen}]{yang2021}
\bibinfo{author}{X.~Yang}, \bibinfo{author}{X.~Zheng},
  \bibinfo{author}{J.~Chen}, \bibinfo{title}{Testing high-dimensional
  covariance matrices under the elliptical distribution and beyond},
  \bibinfo{journal}{Journal of Econometrics} \bibinfo{volume}{221}
  (\bibinfo{year}{2021}) \bibinfo{pages}{409--423}.

\end{thebibliography}

\end{document}